\documentclass[a4paper,reqno,12pt]{amsart}

\usepackage[utf8]{inputenc}
\usepackage{amsmath, amssymb, amsthm, epsfig}
\usepackage{hyperref, latexsym}
\usepackage{url}
\usepackage[mathscr]{euscript}
\usepackage{color}
\usepackage{harpoon}
\usepackage{url}
\usepackage{mathabx}
\usepackage{tikz}
\usepackage{mathtools,fbox}
\usepackage{refcheck}
\DeclareMathAlphabet{\pazocal}{OMS}{zplm}{m}{n}

\usepackage{pgfplotstable}
\usepackage{pgfplots}

\usepackage{fullpage} 
\usepackage{setspace}
\usepackage{mathtools}
\mathtoolsset{showonlyrefs}

\usepackage{refcheck} 
\norefnames
\nocitenames

\def\today{\ifcase\month\or
  January\or February\or March\or April\or May\or June\or
  July\or August\or September\or October\or November\or December\fi
  \space\number\day, \number\year}

\DeclareMathOperator{\supp}{\mathrm{supp}}

\newtheorem{theorem}{Theorem}

\newtheorem{lemma}[theorem]{Lemma}
\newtheorem{proposition}[theorem]{Proposition}
\newtheorem{corollary}[theorem]{Corollary}

\newtheorem{remark}{Remark}

\newtheorem{definition}{Definition}

\newcommand{\A}{\mathcal{A}}

\newcommand{\F}{\mathcal{F}}
\newcommand{\G}{\mathcal{G}}

\newcommand{\M}{\mathcal{M}}
\newcommand{\N}{\mathcal{N}}

\renewcommand{\P}{\mathcal{P}}

\newcommand{\R}{\mathcal{R}}
\renewcommand{\S}{\mathcal{S}}

\newcommand{\U}{\mathcal{U}}

\newcommand{\W}{\mathcal{W}}

\newcommand{\MM}{\mathfrak{M}}
\newcommand{\NN}{\mathfrak{N}}

\newcommand{\n}{\mathbb{N}}
\newcommand{\z}{\mathbb{Z}}
\newcommand{\q}{\mathbb{Q}}
\renewcommand{\r}{\mathbb{R}}
\newcommand{\cp}{\mathbb{C}} %<--- because of my last name
\newcommand{\im}{{\rm Im}\,}

\newcommand{\ft}{\widehat}

\newcommand{\bo}{\boldsymbol}

\newcommand{\la}{\lambda}
\newcommand{\ga}{\gamma}
\newcommand{\al}{\alpha}

\newcommand{\ep}{\varepsilon}

\newcommand{\si}{\sigma}
\newcommand{\p}{\varphi}

\newcommand{\ord}{{\rm ord}}

\renewcommand{\d}{\mathrm{d}}
\newcommand{\ov}[1]{\overline{#1}}
\newcommand{\w}{\ov{w}}
\newcommand{\cd}{\cdot}
\newcommand{\1}{{\bf 1}}
\newcommand{\ra}{\rightarrow}

\newcommand{\e}{\mathbb{E}}
\newcommand{\mt}{\mapsto}
\newcommand{\del}{{\bo \delta}}

\renewcommand{\Re}{{\rm Re}\,}
\renewcommand{\Im}{{\rm Im}\,}
\newcommand{\res}[1]{ \mathop{\mathrm{Res}}_{z = #1}\,}
\newcommand{\spec}{{\rm spec}}

\begin{document}

%------------------HEADINGS------------------------

\title[]{Classification of Fourier summation formulas on a horizontal strip}
\author{Guilherme Vedana}
\date{\today}
\address{IMPA - Instituto de Matemática Pura e Aplicada, Rio de Janeiro, Brazil.}
\email{guilhermeisraelvedana@gmail.com}
\allowdisplaybreaks
%\numberwithin{theorem}{subsection}{section}
%\numberwithin{equation}{section}

%------------------ABSTRACT------------------------------

\begin{abstract}
We classify Fourier summation identities in which the measure on the Fourier side is supported in a horizontal strip of $\mathbb{C}$. Let $\mu=\nu+\eta$, where $\nu$ is a strongly tempered measure on $\mathbb{R}$, $\eta=\sum_{m\geq1} b(\gamma_m)\bo \delta_{\gamma_m}$ is a strongly tempered pure point measure supported off the real line, and $a:\Lambda\rightarrow\mathbb{C}$, with $\Lambda=\{\lambda_n\}_{n\geq1}\subset\r$, has finite exponential growth. Under natural real-antipodal and conjugation-symmetry assumptions, we characterize summation identities of the form
\begin{align}
    \sum_{n\geq1} a(\lambda_n)\varphi(\lambda_n)=\int_{\mathbb{R}} \widehat{\varphi}(t)\mathrm{d}\nu(t)+\sum_{m\geq1} b(\gamma_m)\widehat{\varphi}(\gamma_m),
\end{align}
valid for every $\varphi\in C^\infty_c(\mathbb{R})$, where $\widehat{\varphi}$ denotes the Fourier transform. We prove that each such identity determines a unique generating function $F$ that is holomorphic and almost periodic in the half-plane above the strip and admits a meromorphic continuation to $\mathbb{C}^+$. The measure $\eta$ encodes the poles and residues of $F$, while $\nu$ describes the boundary behavior of its regular part through a generalized Nevanlinna representation, and $a$ determines its Fourier coefficients. Conversely, every function in the corresponding meromorphic class whose Fourier coefficients satisfy a local summability condition determines a unique summation identity of this form. The proof combines a strip version of the Bridge Lemma with a Cauchy-transform argument that accounts for the off-real poles. As an application, we show that the Guinand-Weil explicit formula for every member of the Selberg class, including non-self-dual members, fits into our framework, and we identify its associated generating function.
\end{abstract}

%---------------------TITLE--------------------------------

\maketitle

%---------------------HAVE--FUN!-----------------------------------

\section{Introduction}
A Fourier summation pair on a strip $\U:=\{z;|\Im{z}|<\sigma<\infty\}$ (FS-pair for short) is a pair $(\mu=\nu+\eta,a)$, where $\nu$ and $\eta$ are strongly tempered measures, with $\nu$ supported on $\r$, $\eta=\sum_{m\geq1} b(\ga_m)\del_{\ga_m}$ is a discrete measure supported in $\U\backslash\r$, and $a:\Lambda\rightarrow\cp$, $\Lambda=\{\la_n\}_{n\geq1}\subset\r$, is a locally summable function such that
\begin{align}
    \sum_{n\geq1} a(\la_n)\p(\la_n)=\int_{\r} \ft{\p}(t)\d\nu(t)+\sum_{m\geq1} b(\ga_m)\ft{\p}(\ga_m),
\end{align}
holds for any function $\p:\r\ra\cp$ that is $C^\infty$ and compactly supported, which will be called \textbf{test functions} throughout this paper. In particular, the Fourier transform $\ft{\p}$ can be extended to an entire function of finite exponential type, hence $\ft{\p}(\ga)$ is well-defined. We will adopt the following normalization of the Fourier transform: For $f\in L^1(\r)$
\begin{align*}
    \ft{f}(\xi):=\int_{\r} f(t)e^{-2\pi it\xi}\d t, \xi\in\r.
\end{align*}

We prove that, under certain natural real-antipodal and conjugation-symmetry assumptions, the pair $(\nu+\eta,a)$ is associated with a unique generating function $F$ in the upper half-plane $\cp^+:=\{z\in\cp;\Im{z}>0\}$ that is meromorphic with simple poles at the points $\{\ga_m\}_{m\geq1}$ and the respective residues $-b(\ga_m)/2\pi i$, and its boundary behavior, in the sense of generalized Nevanlinna representations, is described by the measure $\nu$. Moreover, $F$ is almost periodic in some half-plane $\Im{z}>c$ with a non-negative spectrum and Fourier coefficients determined by the function $a(\cd)$. Conversely, starting from any such function $F$ we can construct an FS-pair on a strip $(\nu+\eta,a)$: the measure $\nu$ encodes the boundary behavior of $F$ in the sense of Nevanlinna representation, the pure point measure $\eta$ encodes the poles and residues of $F$, and $a(\cd)$ is obtained from the Fourier coefficients of the almost periodic map $F$. All these definitions and the precise classification results are presented in full detail in the next section.

The passage from real-line support to strip support is not merely a perturbative extension of the real-line classification results in \cite{G23,GV}, where the case $\eta=0$ was considered. Indeed, the off-real atoms of $\eta$ give rise to poles of $F$ in $\cp^+$, so the boundary Nevanlinna representation alone no longer determines the full summation pair. Moreover, our new theorem proves that even when the off-real atoms produce poles that may approach the boundary, the generalized Nevanlinna component $\nu$ and the polar component $\eta$ remain separately and uniquely recoverable. To account for such atoms, we prove a new version of the Bridge Lemma that constructs $F$, and use a Cauchy-transform argument to recover the summation formula.

As an application of our results, we show that the corresponding Guinand-Weil explicit formula of every function in the Selberg class, including the non-self-dual case, fits into our framework, and we compute explicitly its associated generating function.

To the best of our knowledge, there is no previous general two-way classification of real-antipodal Fourier summation pairs whose measure is supported on a horizontal strip, allowing both a possibly non-discrete strongly tempered component on the real line and a discrete off-real component whose weights are not restricted to multiplicities. The results of the present paper provide such a classification in terms of meromorphic continuations of holomorphic almost-periodic functions and generalized Nevanlinna representations.

The works most closely related to the setting of the present paper are due to Favorov. In \cite{Fstrip}, he studied a particular class of discrete measures with unit masses supported on a horizontal strip and established a correspondence between their supports and the zero sets of exponential polynomials or, more generally, of absolutely convergent Dirichlet series with bounded spectrum. More recently, in \cite{FEntire}, Favorov considered the zero-counting measure of an entire almost-periodic function of exponential growth whose zeros lie in a horizontal strip. He proved that its generalized Fourier transform is a pure point measure on $\r$, with weights determined by the coefficients in the Fourier--Dirichlet expansions of the logarithmic derivative, and obtained a necessary and sufficient condition for the function to be, up to an exponential factor, a finite product of sine functions. Favorov and De\v{g}er \cite{FD} subsequently investigated growth properties of Fourier quasicrystals and measures supported on a strip. Although these works are closely related to our setting, our framework concerns general strongly tempered FS-pairs, permits arbitrary complex weights and a non-discrete component supported on the real line, and incorporates the latter through generalized Nevanlinna representations.

Finally, we give a brief overview of some recent results on crystalline measures and Fourier quasicrystals. We recall that a \textit{crystalline measure} on $\r$ is a complex-valued Borel measure of the form $\mu=\sum_{\ga\in\Gamma} a_\ga\del_\ga$, that defines a tempered distribution and its Fourier transform has the same form $\ft{\mu}=\sum_{\la\in\Lambda} b_\la\del_\la$, where $\Gamma,\Lambda\subset\r$ are discrete subsets, and $\del_\ga$ is the delta mass at $\ga$. We call $\mu$ a \textit{Fourier quasicrystal} (FQ) when both $|\mu|$ and $|\ft{\mu}|$ are also tempered distributions.

In 2015, Lev and Olevskii \cite{LO2} characterized the crystalline measures $\mu$ for which the supports of both $\mu$ and $\ft{\mu}$ are uniformly discrete\footnote{A subset $X\subset\r$ is uniformly discrete if $\inf_{x\neq y\in X} |x-y|\geq\delta>0$.} subsets of $\r$. Subsequently, Kurasov and Sarnak \cite{KS} presented a new way of generating FQs from Lee--Yang polynomials. Given a Lee--Yang polynomial\footnote{$P$ is Lee--Yang if it has no roots in $\mathbb{D}^n$ or $\left(\cp\backslash\ov{\mathbb{D}}\right)^n$, where $\mathbb{D}=\{z\in\cp;|z|<1\}$.} $P\in\cp[z_1,\ldots,z_n]$ and a vector $(\ell_1,\ldots,\ell_n)\in\r_+^n$, the function $f(x)=P(e^{ix\ell_1},\ldots,e^{ix\ell_n})$ is a real-rooted trigonometric polynomial\footnote{By a trigonometric polynomial, we mean a map $f(z)=\sum_{n=1}^{N} b_n e^{2\pi i\theta_nz}$, where $b_n\in\cp$ and $\theta_n\in\r$.}, and the measure
    \begin{equation}
    \label{eq:Kurasov_Sarnak_measure}
        \mu=\sum_{\ga;f(\ga)=0} m(\ga)\del_\ga
    \end{equation}
is an $\n$-valued FQ, where $m(\ga)$ denotes the multiplicity of the zero $\ga$. Subsequently, Olevskii and Ulanovskii \cite{OU} proved that the measures of the form \eqref{eq:Kurasov_Sarnak_measure}, with $f$ a real-rooted trigonometric polynomial, include all $\n$-valued FQs (see also \cite[Theorem 5]{G23} for a generalization in which the $\n$-valued condition is removed). Alon, Cohen, and Vinzant \cite{ACV} completed the characterization of one-dimensional $\n$-valued FQs by proving that every real-rooted trigonometric polynomial $f$ can be written as the restriction of a Lee--Yang polynomial, $f(x)=e^{i\la x}P(e^{ix\ell_1},\ldots,e^{ix\ell_n})$, for some $\la\in\r$, an integer $n>0$, a Lee--Yang polynomial $P\in\cp[z_1,\ldots,z_n]$, and a vector $\bo\ell\in\r_+^n$ whose entries are linearly independent over $\q$. Higher-dimensional versions of these results can be found in \cite{AKKV25,LT25}.

Fourier summation formulas restricted to the real-line, i.e., when $\eta=0$, were studied and classified in \cite{G23,GV}. Notable developments also include the new summation formulas of Bondarenko, Radchenko, and Seip \cite{BRS}, Kulikov, Nazarov, and Sodin \cite{KNS}, Radchenko and Viazovska \cite{RV}, and Ramos and Sousa \cite{RS,RS1}; the theory of Lee--Yang polynomials developed by Alon, Cohen, and Vinzant \cite{ACV,AV}; higher-dimensional FQs studied by Alon et al. and by Lawton and Tsikh \cite{AKKV25,LT25}; and the classification of crystalline pairs by Lev and Olevskii \cite{LO,LO2,LO3}, Meyer \cite{Mey23,Mey,Mey17}, and Olevskii and Ulanovskii \cite{OU,OU2}.

\section{Main results}

We begin by introducing some preliminary definitions. A complex-valued Borel measure $\mu$ supported on a strip $\{z\in\cp;|\Im{z}|<\sigma_\mu\}$, with $0\leq\si_\mu<\infty$, will be called \textbf{strongly tempered} if there exists an integer $k\geq0$, such that
\begin{equation}
    \int_{|\Im{z}|<\si_\mu} \frac{\d|\mu|(z)}{(1+|z|^2)^{k/2}}<\infty,
\end{equation}
where $|\mu|$ denotes the total variation of the measure $\mu$. In this case, we call the minimum of such $k$ the \textbf{order} of $\mu$:
\begin{equation}
    \ord{(\mu)}:=\min{\left\{k\in\z_{\geq0};\int_{|\Im{z}|<\si_\mu} \frac{\d|\mu|(z)}{(1+|z|^2)^{k/2}}<\infty\right\}}.
\end{equation}
In particular, the measure $\mu$ has at most countably many atoms\footnote{An atom of a measure $\mu$ is a singleton $\{x\}$ such that $\mu(\{x\})\neq0$.}. In this paper, we consider only strongly tempered measures $\mu$ on a strip that can be written as $\mu=\nu+\eta$, where $\nu$ is a strongly tempered measure on $\r$ and $\eta=\sum_{\ga\in A} b(\ga)\del_\ga$ is a strongly tempered discrete measure such that $A\cap\r=\varnothing$ and all finite limit points of $A$, if any, lie on $\r$. Thus, $\eta$ contains only point masses outside the real line; these may have limit points on $\r$, but not in $\cp\backslash\r$. The decomposition $\mu=\nu+\eta$ is then unique, and $\nu\perp\eta$. Moreover, $\mu$ is strongly tempered if and only if both $\nu$ and $\eta$ are strongly tempered, and $\ord(\mu)\leq k$ if and only if $\ord(\nu),\ord(\eta)\leq k$. We call such a measure $\mu$ \textbf{admissible}; whenever we write $\mu=\nu+\eta$, this decomposition is understood unless explicitly stated otherwise.

\begin{remark}
\label{remark_measure_decomposition}
The requirement $A\cap\r=\varnothing$ is only a convention concerning the decomposition $\mu=\nu+\eta$: all atoms of $\mu$ supported on $\r$ are included in $\nu$, while $\eta$ consists only of the atoms outside $\r$. This convention makes the decomposition unique and simplifies the statements of our main results. It imposes no additional restriction on the total measure $\mu$, the corresponding FS-pair, or its generating function $F$. In particular, when a discrete measure arising in an application has atoms on $\r$, these atoms are assigned to $\nu$ rather than to $\eta$, leaving $\mu$, the FS-pair, and $F$ unchanged.
\end{remark}

Given a test function $\p\in C^\infty_c(\r)$, suppose that $\supp{\p}\subset [-L,L]$. Then, for $t\in\r$, we have 
\begin{equation}
\label{eq:phi_entire_extension}
    \ft{\p}(t)=\int_{-L}^{L} \p(x)e^{-2\pi ixt}\d x.
\end{equation}
A simple application of Morera's theorem shows that, upon replacing $t\in\r$ by $z\in\cp$ in the above expression, the map $\ft{\p}$ extends to an entire function on $\cp$. For simplicity, we use $\ft{\p}$ to denote both the Fourier transform on $\r$ and its entire extension.

A function $a:\r\ra\cp$ will be called \textbf{locally summable} if its support $\supp{a}:=\{\la\in\r;a(\la)\neq0\}$ is a countable set, and for some enumeration $\supp{a}=\{\la_n;n\geq1\}$, the series 
\begin{equation}
    \sum_{n;\, \la_n\in[-T,T]} |a(\la_n)|
\end{equation}
is convergent for any $T>0$. In particular, the sum $\sum_{n\geq1} a(\la_n)\p(\la_n)$ is also absolutely convergent for any $\p\in C^\infty_c(\r)$. The function $a(\cd)$ will be said to be of \textbf{finite exponential growth} if there exists some $\al>0$ for which
\begin{equation}
    \sum_{n\geq1} |a(\la_n)|e^{-\al|\la_n|}<\infty.
\end{equation}
These properties are independent of the enumeration of $\supp{a}$. We therefore use abbreviated notation such as
\begin{align}
    &\sum_{\la\in\r} a(\la)\p(\la):=\sum_{n\geq1} a(\la_n)\p(\la_n),\text{ and } &\sum_{\la>0} a(\la)e^{2\pi i\la z}:=\sum_{n\geq1;\, \la_n>0} a(\la_n)e^{2\pi i\la_n z}.
\end{align}
We are now ready to state the definition of a Fourier summation pair on a strip:

\begin{definition}[Fourier summation pair on a strip]
A Fourier summation pair on a strip is a pair $(\mu,a)$ that satisfies the following properties
\begin{itemize}
    \item[i)] $a:\r\ra\cp$ is a locally summable function.
    \item[ii)] $\mu$ is an admissible, strongly tempered measure supported on a strip $\{z\in\cp;|\Im{z}|<\si_\mu<\infty\}$.
    \item[iii)] For any test function $\p\in C^\infty_c(\r)$ \begin{align}
\label{eq:FSpair_short_def}
    \sum_{n\geq1} a(\la_n)\p(\la_n)=\int_{|\Im{z}|<\si_\mu} \ft{\p}(z)\d\mu(z).
\end{align}
\end{itemize}
\end{definition}
Equivalently, we can write \eqref{eq:FSpair_short_def} as
\begin{align}
\label{eq:FSpair_def}
    \sum_{n\geq1} a(\la_n)\p(\la_n)=\int_{\r} \ft{\p}(t)\d\nu(t)+\sum_{\ga\in A} b(\ga)\ft{\p}(\ga),
\end{align}
for the decomposition $\mu=\nu+\eta$. In order to ease the notation, we will refer to `FS-pairs on a strip' as just `FS-pairs'. Moreover, sometimes we will simply say that $(\mu,a)$ is an FS-pair on the closed strip $|\Im{z}|\leq\sigma$. In this case, we mean that $(\mu,a)$ is an FS-pair in the above definition on the open strip $|\Im{z}|<\sigma+\ep$ for all $\ep>0$.

To simplify the statements of our main results, we restrict attention to a subclass of FS-pairs on a strip. We say that the discrete measure $\eta=\sum_{\ga\in A} b(\ga)\del_\ga$ is \textbf{symmetric with respect to $\r$}, or $\r$-symmetric, if $\ga\in A$ precisely when $\ov{\ga}\in A$ and $b(\ov{\ga})=b(\ga)$ for any $\ga\in A$. We call $(\mu=\nu+\eta,a)$ an $\r$-symmetric FS-pair when $\eta$ is $\r$-symmetric. We consider such FS-pairs for two reasons: the Guinand--Weil explicit formulas for functions in the Selberg class have this symmetry, and it allows us to state our results for real-antipodal FS-pairs (see below).
\subsection*{Real antipodal splitting} A function $a:\r\ra\cp$ is called \textbf{antipodal} if $\ov{a(\la)}=a(-\la)$ for every $\la\in\r$. We call an FS-pair on a strip $(\mu,a)$ \textbf{real-antipodal} if $\mu$ is real-valued and $a(\cdot)$ is antipodal. Any $\r$-symmetric FS-pair $(\mu=\nu+\eta,a)$ can be split into two $\r$-symmetric FS-pairs $(\mu_1,a_1)$ and $(\mu_2,a_2)$ on the same strip. Namely, set $\mu_1:=\Re{\mu}=\Re{\nu}+\Re{\eta}$, $\mu_2:=-\Im{\mu}=-\Im{\nu}-\Im{\eta}$, $a_1(\la):=\left(a(\la)+\ov{a(-\la)}\right)/2$, and $a_2(\la):=-i\left(\ov{a(-\la)}-a(\la)\right)/2$. To verify that $(\mu_j,a_j)$ is an FS-pair, by linearity it suffices to consider antipodal test functions $\p$. For such functions, $\ft{\p}(w)+\ft{\p}(\ov{w})$ is real-valued; combining this with $b(\ov{\ga})=b(\ga)$ proves the claim. The original pair $(\mu,a)$ can be recovered from $(\mu_1,a_1)$ and $(\mu_2,a_2)$. We may therefore state our results only for $\r$-symmetric, real-antipodal FS-pairs.

We now briefly introduce two special classes of functions that will play a key role in our main results. We provide more details in Sections \ref{section:almost_periodic_class} and \ref{section:Nevanlinna_class}.
\subsection*{The almost periodic class} We say that a continuous function $f:\r\ra\cp$ is \textbf{almost periodic} if, for any $\ep>0$, there exists a set of translations $\tau_\ep\subset\r$, which is relatively dense\footnote{This means that there exists $l>0$ such that $\tau\cap(x,x+l)\neq\emptyset$ for any $x\in\r$.}, and $\sup_{x\in\r} {|f(x+t)-f(x)|}<\ep$, for any $t\in\tau_{\ep}$. The class of \textbf{holomorphic almost periodic} functions on the upper half-plane $\cp^+$ is the set of holomorphic maps $F$ on $\cp^+$ such that $F$ is bounded on strips, i.e., 
\begin{align*}
    \sup_{\al<\Im{z}<\beta} |F(z)|<\infty
\end{align*}
for any $0<\al<\beta<\infty$, and for each $y>0$, the function $x\mt F(x+iy)$ is almost periodic when seen as a map defined on $\r$. We denote this class by $\A\P(\cp^+)$.
For $F\in\A\P(\cp^+)$, we can define an analog of a Fourier coefficient\footnote{Throughout this paper, we will call $\e F(\la)$ the Fourier coefficient of $F$. In addition, any exponential series of the form $\sum_{n\geq0} a_n e^{2\pi i\la_n z}$, with frequencies $\la_n\in\r$, will be called a 'Fourier series'.}, also known as `Bohr-Fourier' coefficient: Given any $\la\in\r$, the limit
\begin{align*}
\e F(\la):=\lim_{T\ra\infty} \frac{1}{2T} \int_{-T+iy}^{T+iy} F(z)e^{-2\pi i\la z}dz
\end{align*}
does exist and is independent of $y>0$. In the special case where $z\mt F(z+ic)$ is an almost periodic function of $z$, the above limit exists and does not depend on $y>c$. The spectrum of $F$ is defined by $\spec(F):=\{\la\in\r;\e F(\la)\neq0\}$, which is a countable set.

\subsection*{The Nevanlinna class} The class of generalized holomorphic Nevanlinna functions of order $\leq k$, which will be denoted by $\NN_{\leq_k}$, is the set of all holomorphic maps $F:\cp^+\ra\cp$ that satisfy the following property: for any $N>0$ and any choice of $z_1,...,z_N \in \cp^+$, the Hermitian matrix
\begin{align*}
\left[i\frac{F(z_n)+\ov{F(z_m)}}{z_n-\ov{z_m}}\right]_{1\leq n,m\leq N}
\end{align*}
has at most $k$ negative eigenvalues (counting multiplicities). Our main results will make use of the set of differences $\NN_{\leq k}-\NN_{\leq k}:=\{F-G;F,G\in\NN_{\leq k}\}$, also known as the space of quasi-Herglotz functions. For simplicity, we will continue to refer to such functions as generalized Nevanlinna functions.

We now discuss the intuition behind our main results and explain their connection with Nevanlinna and almost periodic maps. To clarify the main ideas, we will be deliberately informal and omit technical details. Let $F:\cp\ra\cp$ be a meromorphic map such that
\begin{itemize}
    \item[1)] $F(z)$ can be represented as
    \begin{align*}
    F(z)=\sum_{n\geq0} b_ne^{2\pi i\la_nz},&\,\, \text{ for } \Im{z}>\al_1\\
    F(z)=\sum_{m\geq0} c_me^{-2\pi i\xi_mz},&\,\, \text{ for } \Im{z}<-\al_2,
    \end{align*}
for some $0<\al_1,\al_2<\infty$, and $\la_n,\xi_m\geq0$. Assume that both expressions converge uniformly in these domains, and
    \item[2)] all singularities of $F(z)$ are simple poles and are contained on a horizontal strip $\U:=\{z;|\Im{z}|<\sigma\}$, for some $\sigma\leq\min\{\al_1,\al_2\}$ . Let $\{\ga_n\}_{n\geq1}$ be the set of such poles.
\end{itemize}

The (normalized) Mellin transform of a function $\psi\in C^\infty_c(0,\infty)$ is $\M\psi(s):=\int_{0}^{\infty} \psi(x)x^{2\pi s}\frac{\d x}{x}$, which is an entire function that rapidly decays as $|\Im(s)|\ra\infty$ for fixed $\Re{s}$. Now, for a test function $\p\in C^\infty_c(\r)$, we consider the integral
\begin{equation}
\label{eq:integral_example}
    \int_{C_R} F(z)\M(\p\circ\log)(-iz)\d z,
\end{equation}
where $C_R$ is the rectangle with vertices $R+2i\al_1$, $-R+2i\al_1$, $-R-2i\al_2$, and $R-2i\al_2$. We evaluate this integral in two ways: using the residue theorem and using the two Fourier expansions of $F$. Letting $R\ra+\infty$ and assuming that the integrals on the vertical segments tend to zero, we obtain\footnote{After multiplying both sides of the equation by $-1$, so that the sign matches the convention used in our main results.}
\begin{equation}
    -2\pi i\sum_{j\geq0} \left(\res{\ga_j}F\right)\ft{\p}(\ga_j)=\sum_{n\geq0} b_n\p(\la_n)-\sum_{m\geq0} c_m\p(-\xi_m),
\end{equation}
that is, $(\mu,a)$ is an FS-pair on the strip $\U$, where $\mu:=-2\pi i\sum_{j\geq0} \left(\res{\ga_j}F\right)\del{\ga_j}$, $a(\la_n):=b_n$, $a(-\xi_m):=-c_m$, $a(0):=b_{n_0}-c_{m_0}$ when $\xi_{m_0}=\la_{n_0}=0$, and $a$ is zero otherwise. Thus, the Fourier representations of $F$ determine $a(\cdot)$, while its singularities determine $\mu$. This construction was used by Kurasov and Sarnak in \cite{KS}, by Favorov in \cite{F24,Fstrip,FEntire} and can also be found in \cite{BRS}.

The same construction works if we start with a meromorphic function $F$ defined on $\cp^+$, with a Fourier series expansion $F(z)=\sum_{n\geq0} b_n e^{2\pi i\la_nz}$ on some half-plane $\Im{z}>\al_1$, and with singularities determined by simple poles located on a strip $0<\Im{z}<\si$ together with its boundary behavior on $\r$. This is the idea of our Theorem \ref{thm:converse_main}. 

Since we are interested in starting with the FS-pair $(\mu,a)$ and obtaining such a generating function $F$, and since $\mu$ encodes the simple poles and boundary behavior of $F$, we would like to find a way to construct $F$ from its singularities. Clearly, there is no unique way to construct such an $F$, since $F+G$ has the same singularities as $F$ for any entire function $G$.

Let us first assume that $F$ is holomorphic on $\mathbb{C}^+$. Under suitable growth conditions, $F$ can be reconstructed from its boundary values. One such setting is the Hardy space $\mathcal{H}^p(\mathbb{C}^+)$, $1\leq p<\infty$, in which a function $F\in\mathcal{H}^p(\mathbb{C}^+)$ is determined by its non-tangential boundary values $F(t)\in L^p(\mathbb{R})$ through the formula $F(z)=\frac{1}{2\pi i}\int_{\r} \frac{F(t)}{t-z}\d t$. It is possible to go further and consider the case in which the boundary behavior of $F$ is represented by a positive strongly tempered measure $\nu$ with $\ord(\nu)\leq 2(k+1)$. In this case, the expression

\begin{equation}
\label{eq:Nevanlinna_motivation}
    F(z)=\frac{(z^2+1)^k}{2\pi i}\int_{\r} \frac{1+tz}{t-z}\frac{d\nu(t)}{(1+t^2)^{k+1}}
\end{equation}
defines a holomorphic function on $\mathbb{C}^+$ whose boundary behavior is encoded by $\nu$. The class of functions that admit a representation of the form \eqref{eq:Nevanlinna_motivation} for such a measure $\nu$ is called the generalized Nevanlinna class of index $\leq k$. This naturally leads to the appearance of the Nevanlinna class in our results.

Observe that we can go further and consider in \eqref{eq:Nevanlinna_motivation} an admissible, strongly tempered measure $\mu=\nu+\eta$ supported on $|\Im{z}|<1$, thus obtaining a meromorphic function

\begin{equation}
    G(z)=\frac{(z^2+1)^k}{2\pi i}\int_{\r} \frac{1+tz}{t-z}\frac{d\nu(t)}{(1+t^2)^{k+1}}+\frac{(z^2+1)^k}{2\pi i}\sum_{\ga\in A} \frac{1+\ga z}{\ga-z}\frac{b(\ga)}{(1+\ga^2)^{k+1}}.
\end{equation}
This is the idea of Theorem \ref{thm:main}. Given an FS-pair on a strip $(\mu=\nu+\eta,a)$ define $F(z):=a(0)/2+\sum_{\la>0} a(\la)e^{2\pi i\la z}$ on a half-plane $\Im{z}>\al_1$. Then $F$ extends to $\cp^+$ as a meromorphic function and, for a suitable real polynomial $Q$, it satisfies $F(z)=G(z)+iQ(z)$. The function $G$ is no longer a Nevanlinna function. For this reason, throughout the text we decompose $G$ into its regular and singular parts. In this way, $F$ encodes the information of the FS-pair: its Fourier coefficients are given by the function $a(\cdot)$, the measure $\eta$ encodes its poles and residues, and the measure $\mu$ encodes its boundary behavior on $\mathbb{R}$.

Conversely, Theorem \ref{thm:converse_main} states that, starting from any such function $F$, one can recover an FS-pair on a strip $(\nu+\eta,a)$. The function $a(\cdot)$ is obtained from the Fourier coefficients of the almost periodic function $F$, while $\eta$ is a sum of point masses supported at the poles of $F$, with coefficients determined by the corresponding residues. Finally, after removing the singular part of $F$, the measure $\nu$ is uniquely determined by the Herglotz-Nevanlinna representation. More precisely, we have the following.
\vspace{3mm}

\begin{theorem}[Classification of FS-pairs on a strip]
\label{thm:main}
Let $(\mu=\nu+\eta,a)$ be a real-antipodal FS-pair on a strip $\{z;|\Im{z}|<\si_\mu\}$, which is $\r$-symmetric, $\ord{(\mu)}\leq 2(k+1)$ for some $k\geq0$, and $a(\cd)$ has finite exponential growth. Then, the expression
\begin{align}
F(z):=\frac{1}{2}a(0)+\sum_{\la>0} a(\la)e^{2\pi i\la z},
\end{align}
defines a holomorphic almost periodic map on the half-plane $\Im{z}>c$, for some $c>0$ depending on the exponential growth of $a(\cd)$. Moreover, $F$ extends as a meromorphic map to $\cp^+$ as $F(z)=N(z)+S(z)$ for some Nevanlinna function $N\in\NN_{\leq k}-\NN_{\leq k}$ and some meromorphic map $S(z)$ given explicitly by 
\begin{equation}
\label{eq:F_formula}
F(z)= \frac{(z^2+1)^k}{2\pi i} \int_{\r} \frac{1+tz}{t-z}\cd\frac{d\nu(t)}{(1+t^2)^{k+1}}+iQ_\rho(z)+\frac{(z^2+\rho^2)^k}{2\pi i} \sum_{\ga\in A} \frac{\rho^2+\ga z}{\ga-z}\cd\frac{b(\ga)}{(\rho^2+\ga^2)^{k+1}}.
\end{equation}
Above, $\rho>\si_\mu$ is any parameter, and $Q_\rho(z)$ is a real polynomial of degree $\leq 2k$, which depends on $\rho$.
\end{theorem}
The next result provides a converse to Theorem \ref{thm:main} and shows how to construct an FS-pair on a strip starting from such generating functions.
\begin{theorem}[Constructing FS-pairs]
\label{thm:converse_main}
Let $F(z)$ be a meromorphic map in $\cp^+$ such that $z\mt F(z+ic)\in\A\P(\cp^+)$ for some $c>0$ and the function $\la\mt\e F(\la)$ is locally summable. Suppose that the singularities of $F$ are all simple poles, denoted by $\ga_n$ and respective residues $-\frac{b(\ga_n)}{2\pi i}$, satisfying $b(\ga_n)\in\r$ and
\begin{equation}
\label{eq:convergence_eta_series}
    \sum_{n\geq0} \frac{|b(\ga_n)|}{(1+|\ga_n|^2)^{m+1}}<\infty
\end{equation}
for some $m\in\z_{\geq0}$. Moreover, assume that $F-S\in\NN_{\leq j}-\NN_{\leq j}$, for some $j\in\z_{\geq0}$, where we define
\begin{equation}
\label{eq:S_representation}
    S(z):=\frac{(z^2+\rho^2)^m}{2\pi i} \sum_{\ga\in A} \frac{\rho^2+\ga z}{\ga-z}\cd\frac{b(\ga)}{(\rho^2+\ga^2)^{m+1}},
\end{equation}
with $A:=\{\ga_n\}_{n\geq0}\cup\{\ov{\ga_n}\}_{n\geq0}$, $b(\ov{\ga_n}):=b(\ga_n)$ and some parameter $\rho>c$.

Then, $(\nu+\eta,a)$ is a real-antipodal $\r-$symmetric FS-pair on the strip $\{z\in\cp;|\Im{z}|\leq c\}$, where $\eta:=\sum_{\ga\in A} b(\ga)\del_{\ga}$, $\nu$ is the (unique) real-valued measure from the Herglotz-Nevanlinna representation of $F-S$ and 
\begin{equation}
\label{eq:def_function_a}
a(\la):=\begin{cases}
\e F(\la) & \text{if } \la>0,\\
2\Re \e F(0) & \text{if } \la=0,\\
\ov{\e F(-\la)} & \text{if } \la<0.
\end{cases}
\end{equation}
If, in addition, $\la\mt\e F(\la)$ has finite exponential growth, then so does $a(\cd)$.
\end{theorem}
\begin{remark}
In Theorem \ref{thm:converse_main}, if $F-S\in\NN_{\leq j}-\NN_{\leq j}$ for one $\rho>c$, then this condition holds for every $\rho>c$. Moreover, the integer $m$ could be replaced for any other $m'\in\z_{\geq0}$ for which \eqref{eq:convergence_eta_series} converges and we would still have $F-S_{m'}\in\NN_{\leq j'}-\NN_{\leq j'}$ for some $j'\in\z_{\geq0}$, where $S_{m'}$ denotes \eqref{eq:S_representation} for $m'$; see the discussion at the end of Section \ref{section:Nevanlinna_class}. Finally, the measure $\nu$ obtained from the Herglotz-Nevanlinna representation of $F-S$ satisfies $\ord(\nu)\leq 2(j+1)$. 
\end{remark}

Now, let $\F\S$ be the class of the $\r-$symmetric real-antipodal FS-pairs on a strip $(\mu,a)$ such that $a(\cd)$ is a function of finite exponential growth. Let $\A\P^*(\cp^++ic),c>0$ be the set of all meromorphic maps $F$ on $\cp^+$ such that $F(\cd+ic)\in\A\P(\cp^+)$, $\e F(0)\in\r$ and $\la\mt\e F(\la)$ is a function of finite exponential growth. For $m\in\z_{\geq0}$, we also denote by $\MM_{\leq m}$ the set of all meromorphic functions $S$ in $\cp^+$ that can be represented as \eqref{eq:S_representation} for some discrete set $A\subset\{z;|\Im{z}|<c_S<\infty\}$, such that the measure $\eta:=\sum_{\ga\in A} b(\ga)\del_\ga$ is real-valued, $\r-$symmetric and strongly tempered with $\ord(\eta)\leq 2(m+1)$, $A\cap\r=\varnothing$ and all finite limit points of $A$, if any, lie on $\r$, and for some parameter $\rho>c_S$. Finally, let
\begin{equation*}
    \G:=\left(\bigcup_{c>0} \A\P^*(\cp^++ic)\right)\cap\left(\bigcup_{k,m\geq0} \left(\NN_{\leq k}-\NN_{\leq k}\right)+\MM_{\leq m}\right),
\end{equation*}
where the addition of sets if defined by $U+V:=\{u+v;u\in U,v\in V\}$.
\begin{corollary}
\label{cor:bijection}
    The map $\F\S\ra\G$, $(\mu,a)\mt F$ given by Theorem \ref{thm:main} is a bijection. Its inverse is the correspondence $F\mt (\mu,a)$ constructed in Theorem \ref{thm:converse_main}.
\end{corollary}
Now, suppose that the measure $\nu=\sum_{n\geq1} p_n\del_{q_n}$, $p_n,q_n\in\r$ is a countable sum of delta masses and that $\supp{\nu}\cup\supp{\eta}$ has no finite limit points. Then Theorems \ref{thm:main} and \ref{thm:converse_main} assume the form:
\begin{corollary}
    \label{corollary_discrete_measure}
    Let $(\mu=\nu+\eta,a)$ be a real-antipodal $\r-$symmetric FS-pair on a strip $|\Im{z}|<\sigma_\mu$, with $\nu$ as above, $\ord(\nu),\ord(\eta)\leq 2(k+1)$ for some $k\in\z_{\geq0}$, and $a(\cd)$ of finite exponential growth. Then the series
    \begin{equation}
    \label{eq:corollary_Fourier_series}
        F(z)=\frac{a(0)}{2}+\sum_{\la>0} a(\la)e^{2\pi i\la z},
    \end{equation}
    converges absolutely in some half-plane $\Im{z}>c$, thus defining a holomorphic almost periodic function there. It extends to $\cp$ as a meromorphic map with simple poles at $\ga\in A$ and $q_n\in\r$ with respective residues $-b(\ga)/2\pi i$ and $-p_n/2\pi i$. The map $F(z)$ satisfies the relation $\ov{F(\ov{z})}=-F(z)$ and it is given explicitly by
    \begin{equation}
    \label{eq:corollary_extension_formula}
        F(z)=\frac{(z^2+1)^k}{2\pi i}\sum_{n\geq1} \frac{1+q_nz}{q_n-z}\frac{p_n}{(1+q_n^2)^{k+1}}+
        \frac{(z^2+\rho^2)^k}{2\pi i}\sum_{\ga\in A} \frac{\rho^2+\ga z}{\ga-z}\frac{b(\ga)}{(\rho^2+\ga^2)^{k+1}}+iQ_\rho(z),
    \end{equation}
    for $\rho>\sigma_\mu$ and for some real polynomial $Q_\rho(z)$ of degree at most $2k$.
    
Conversely, let $F:\cp\rightarrow\cp$ be a meromorphic map given by \eqref{eq:corollary_extension_formula}, with $p_n,q_n,b(\ga)\in\r$, $A\subset\{z;|\Im{z}|<\sigma<\infty\}$, $A=\{\ov{z};z\in A\}$, $b(\ga)=b(\ov{\ga})$ satisfying $\sum_{n\geq1} \frac{|p_n|}{(1+q_n^2)^{k+1}}<\infty$ and $\sum_{\ga\in A} \frac{|b(\ga)|}{(1+|\ga|^2)^{k+1}}<\infty$. Moreover, assume that $F(\cd+ic)\in\A\P(\cp^+)$, for some $c>0$, and $\la\mt\e F(\la)$ is locally summable. Then $(\nu+\eta,a)$ is a real-antipodal $\r-$symmetric FS-pair on the strip $|\Im{z}|<\sigma$, where
 \begin{align*}
     \nu:=\sum_{n\geq1} p_n\del_{q_n},\quad \eta:=\sum_{\ga\in A} b(\ga)\del_\ga,\quad a(\la)=\ov{a(-\la)}:=\e F(\la),\quad \text{ and } a(0)=2\Re{\e F(0)}
 \end{align*}
for $\la>0$, and $\ord(\nu),\ord(\eta)\leq 2(k+1)$.
\end{corollary}

Finally, we present an application of our results to the Guinand-Weil explicit formula for the Selberg class. Let $L\in\S$ given by
\begin{align*}
    L(s)=\sum_{n\geq1} \frac{a_L(n)}{n^s}, \quad \Re{s}>1.
\end{align*}
Then the explicit-formula for $L$ takes the following form: for any $\p\in C^\infty_c(\r)$
\begin{align*}
     &\frac{1}{2\pi}\sum_{n\geq2}\frac{1}{\sqrt n}
    \left\{
    \Lambda_L(n)\p\left(\frac{\log n}{2\pi}\right)
    +\ov{\Lambda_L(n)}\p\left(-\frac{\log n}{2\pi}\right)
    \right\}\\
    &=\frac{1}{2\pi}\int_{\r}\ft{\p}(t)W_L(t)\d t+m_L\left\{\ft{\p}\left(\frac{i}{2}\right)
    +\ft{\p}\left(-\frac{i}{2}\right)\right\}
    -\sum_{\rho}\ft{\p}(\ga_\rho),
\end{align*}
for some $Q_L>0,\la_j>0, \Re{\mu_j}\geq0$ depending on $L$, $\ga_\rho=(1/2-\rho)/i$, where $\rho$ runs over the non-trivial zeros of $L$ (counting multiplicities),
\begin{align*}
    W_L(t):=2\log Q_L+2\Re\sum_{j=1}^r
    \lambda_j\frac{\Gamma'}{\Gamma}
    \left(\lambda_j\left(\frac12-it\right)+\mu_j\right)\quad \text{and} \quad -\frac{L'}{L}(s)=\sum_{n\geq2} \frac{\Lambda_L(n)}{n^s}.
\end{align*}
We now set

\begin{align*}
    a\left(\frac{\log n}{2\pi}\right)
    &=\frac{\Lambda_L(n)}{2\pi\sqrt n},
    &
    a\left(-\frac{\log n}{2\pi}\right)
    &=\frac{\ov{\Lambda_L(n)}}{2\pi\sqrt n},\\
    \d\nu(t)
    &=\frac{W_L(t)}{2\pi}\d t
      -\sum_{\substack{\rho\\ \Re\rho=1/2}}\del_{\ga_\rho},
    &
    \eta
    &=m_L\left(\del_{i/2}+\del_{-i/2}\right)
      -\sum_{\substack{\rho\\ \Re\rho\neq1/2}}\del_{\ga_\rho},
\end{align*}
for $n\geq2$, and $a(x)=0$ otherwise, and $m_L\in\z_{\geq0}$ is the order of the possible pole of $L$ at $s=1$. The symmetry and distribution of zeros of $L$ imply that $(\nu+\eta,a)$ is an $\r-$symmetric real-antipodal FS-pair on the strip $|\Im{z}|\leq1/2$ with $\ord(\nu),\ord(\eta)\leq 2$ and $a(\cd)$ of finite exponential growth. Applying Theorem \ref{thm:main} with $k=0$ and auxiliary parameter equal to $1$, we obtain
\begin{corollary}
    \label{cor:Selberg_example}
    The generating function given by Theorem \ref{thm:main} of the FS-pair $(\nu+\eta,a)$ defined above by the explicit formula for $L$ is
    
    \begin{equation}
    F_L(z)=\frac{1}{2\pi}\sum_{n\geq2}
    \frac{\Lambda_L(n)}{n^{1/2-iz}}
    =-\frac{1}{2\pi}\frac{L'}{L}\left(\frac12-iz\right),
    \qquad \Im z>\frac12,
\end{equation}
which extends to $\cp^+$ as a meromorphic function through the expression

\begin{align}
\label{eq:Selberg_log_derivative_extension}
    F_L(z)=&\frac{1}{4\pi^2 i}\int_{\r}
    \frac{1+tz}{t-z}W_L(t)\frac{\d t}{1+t^2}
    -\frac{4m_Lz}{i\pi(1+4z^2)}\\
    &-\frac{1}{2\pi}\sum_{\rho}
    \frac{1+i(\rho-1/2)z}
    {\bigl((\rho-1/2)^2-1\bigr)\bigl((\rho-1/2)+iz\bigr)}
    +iC_L,
\end{align}
where $C_L=-\frac{1}{2\pi}\Im\frac{L'}{L}\left(\frac32\right)\in\r$.
\end{corollary}
The expression \eqref{eq:Selberg_log_derivative_extension} gives an explicit meromorphic continuation of $F_L$ in terms of a Herglotz-Nevanlinna representation. Although meromorphic extension formulas for the logarithmic derivative $L'/L$ were already known in the literature, to the best of our knowledge this Herglotz-Nevanlinna normalization has not been previously recorded.

The paper is organized as follows: In Section \ref{section:applications} we recall the definition and some properties of the Selberg class and prove Corollary \ref{cor:Selberg_example}. In Section \ref{sec:aux_results} we collect the essential facts about almost periodic and Nevanlinna functions, and we also prove a Bridge Lemma for FS-pairs on a strip. Theorems \ref{thm:main} and \ref{thm:converse_main} and Corollary \ref{cor:bijection} are proved in Section \ref{section:proof_main_results}.

\section{Explicit formula for the Selberg class}\label{section:applications}

Following \cite{KP}, we recall the definition of the Selberg class $\S$. Every $L\in\S$ satisfies the following axioms: $L$ is a Dirichlet series
\begin{equation*}
    L(s)=\sum_{n\geq1} \frac{a_L(n)}{n^s}
\end{equation*}
that converges absolutely for $\Re{s}>1$ and $a_L(n)\ll n^\ep$ for any $\ep>0$\footnote{We recall that $f(x)\ll g(x)$ for $x\ra\infty$ means that, for some finite constant $C$, $|f(x)|\leq C|g(x)|$ holds for all sufficiently large $x$.}. There is an integer $m_L\geq0$ such that $(s-1)^{m_L}L(s)$ is entire of finite order; we take $m_L$ to be the order of the possible pole at $s=1$. $L$ satisfies a functional equation of the form $\Phi_L(s)=\omega_L\ov{\Phi_L(1-\ov{s})}$, where
\begin{equation*}
\Phi_L(s):=Q_L^s\prod_{j=1}^r\Gamma(\lambda_js+\mu_j)L(s),
\end{equation*}
with $Q_L>0$, $\lambda_j>0$, $\Re\mu_j\geq0$, and $\omega_L\in\cp$, $|\omega_L|=1$. Finally,
\begin{equation*}
    \log L(s)=\sum_{n\geq2}\frac{b_L(n)}{n^s},
\end{equation*}
where $b_L(n)=0$ unless $n=p^m$ for some prime $p$ and $m\geq1$, and $b_L(n)\ll n^{\theta}$ for some $\theta<1/2$.

We define the generalized von Mangoldt function by $\Lambda_L(n):=b_L(n)\log n$, so that
\begin{equation}
\label{eq:log_derivative_Selberg}
    -\frac{L'}{L}(s)=\sum_{n\geq2}\frac{\Lambda_L(n)}{n^s}
\end{equation}
holds with absolute convergence on $\Re{s}>1$.

The completed function
\begin{equation*}
    \xi_L(s):=s^{m_L}(1-s)^{m_L}\Phi_L(s)
\end{equation*}
is entire of order 1. The zeros $\rho$ of $\xi_L$ are located in the closed critical strip $0\leq\Re{s}\leq1$ and are called non-trivial zeros of $L$. The functional equation pairs $\rho$ with $1-\ov{\rho}$. In contrast to the Riemann zeta function $\zeta(s)$, an element $L\in\S$ may have a zero at $s=1$. In this case, $m_L=0$ and $\rho=1$ is a non-trivial zero, and it is paired with $\rho=0$.

Set
\begin{equation*}
    W_L(t):=2\log Q_L+2\Re\sum_{j=1}^r
    \lambda_j\frac{\Gamma'}{\Gamma}
    \left(\lambda_j\left(\frac12-it\right)+\mu_j\right).
\end{equation*}
Then the Guinand-Weil explicit formula for $L$ takes the form
\begin{align}
\label{eq:Selberg_explicit_formula}
    &\frac{1}{2\pi}\sum_{n\geq2}\frac{1}{\sqrt n}
    \left\{
    \Lambda_L(n)\p\left(\frac{\log n}{2\pi}\right)
    +\ov{\Lambda_L(n)}\p\left(-\frac{\log n}{2\pi}\right)
    \right\}\\
    &=\frac{1}{2\pi}\int_{\r}\ft{\p}(t)W_L(t)\d t+m_L\left\{\ft{\p}\left(\frac{i}{2}\right)
    +\ft{\p}\left(-\frac{i}{2}\right)\right\}
    -\sum_{\rho}\ft{\p}(\ga_\rho),
\end{align}
for every $\p\in C_c^\infty(\r)$, where $\ga_\rho:=\frac{1/2-\rho}{i}$ and $\rho$ runs over the non-trivial zeros of $L$, counted with multiplicity. See \cite[Lemma~9, Eq.~(6.1)]{PS}\footnote{Although \cite[Lemma~9, Eq.~(6.1)]{PS} is stated under the assumption that $s=1$ is not a zero of $L$, the same residue argument applies without this assumption when the endpoint zeros are included in the zero sum.}.

We now define
\begin{align*}
    a\left(\frac{\log n}{2\pi}\right)
    &=\frac{\Lambda_L(n)}{2\pi\sqrt n},
    &
    a\left(-\frac{\log n}{2\pi}\right)
    &=\frac{\ov{\Lambda_L(n)}}{2\pi\sqrt n},\\
    \d\nu(t)
    &=\frac{W_L(t)}{2\pi}\d t
      -\sum_{\substack{\rho\\ \Re\rho=1/2}}\del_{\ga_\rho},
    &
    \eta
    &=m_L\left(\del_{i/2}+\del_{-i/2}\right)
      -\sum_{\substack{\rho\\ \Re\rho\neq1/2}}\del_{\ga_\rho},
\end{align*}
for $n\geq2$, and $a(x)=0$ otherwise, and $\d t$ denotes the Lebesgue measure. It follows that $a(\cd)$ is antipodal and of finite exponential growth. Moreover, $\ord(\nu),\ord(\eta)\leq2$, which follows from Stirling's formula and the fact that $\xi_L$ has order 1. Finally, the symmetry of the zeros $\rho$ implies that $\eta$ is $\r$-symmetric. Therefore, $(\nu+\eta,a)$ defines an FS-pair on the strip $|\Im{z}|\leq1/2$, which is real-antipodal and $\r$-symmetric.

In accordance with the convention in Remark \ref{remark_measure_decomposition}, the zeros for which $\ga_\rho\in\r$ contribute point masses to $\nu$, while the remaining zeros contribute point masses to $\eta$. This choice concerns only the decomposition $\mu=\nu+\eta$ and does not affect the total measure $\mu$ or the associated generating function $F_L$.

Applying Theorem \ref{thm:main} with $k=0$ and auxiliary parameter equal to 1, we obtain the generating function
\begin{equation}
\label{eq:Selberg_generating_function}
    F_L(z)=\frac{1}{2\pi}\sum_{n\geq2}
    \frac{\Lambda_L(n)}{n^{1/2-iz}}
    =-\frac{1}{2\pi}\frac{L'}{L}\left(\frac12-iz\right),
    \qquad \Im z>\frac12,
\end{equation}
and its meromorphic extension to $\cp^+$ is given by
\begin{align}
\label{eq:Selberg_meromorphic_continuation}
    F_L(z)={}&\frac{1}{4\pi^2 i}\int_{\r}
    \frac{1+tz}{t-z}W_L(t)\frac{\d t}{1+t^2}
    -\frac{4m_Lz}{i\pi(1+4z^2)}\\
    &-\frac{1}{2\pi}\sum_{\rho}
    \frac{1+i(\rho-1/2)z}
    {\bigl((\rho-1/2)^2-1\bigr)\bigl((\rho-1/2)+iz\bigr)}
    +iC_L,
\end{align}
where $C_L=-\frac{1}{2\pi}\Im\frac{L'}{L}\left(\frac32\right)\in\r$, which follows by computing $F_L(i)$ and comparing the real and imaginary parts.

Observe that if $L$ is self-dual, that is, if $L(s)=\ov{L(\ov{s})}$, or equivalently if its Dirichlet coefficients are real, then $(L'/L)(3/2)\in\r$ and hence $C_L=0$. Self-duality is sufficient, but not necessary, for the vanishing of $C_L$.

In the particular case where $L=\zeta$ is the Riemann zeta function, then the Guinand-Weil explicit formula assumes the simpler version
\begin{align*}
    \frac{1}{2\pi}\sum_{n\geq1} \frac{\Lambda(n)}{\sqrt{n}}\left(\p\left(\frac{\log{n}}{2\pi}\right)+\p\left(-\frac{\log{n}}{2\pi}\right)\right)=&\frac{1}{2\pi}\int_{\r} \ft{\p}(t)\left\{\Re{\frac{\Gamma'}{\Gamma}\left(\frac{1}{4}+i\frac{t}{2}\right)}-\log{\pi}\right\}\d t\\
    &+\ft{\p}\left(\frac{i}{2}\right)+\ft{\p}\left(-\frac{i}{2}\right)-\sum_{\rho} \ft{\p}\left(\ga_\rho\right),
\end{align*}
which holds for any test function $\p\in C^\infty_c(\r)$, and $\Lambda(n)$ is the Von-Mangoldt function defined by
\begin{equation}
    \Lambda(n)=\begin{cases}
        \log{p}, & \text{ if } n=p^k,\, k\geq1\\
        0, & \text{ otherwise}.
    \end{cases}
\end{equation}
The corresponding FS-pair $(\nu+\eta,a)$ on the strip $|\Im{z}|\leq1/2$ is given by
\begin{align*}
    &a\left(\pm\frac{\log{n}}{2\pi}\right)=\frac{1}{2\pi}\frac{\Lambda(n)}{\sqrt{n}},\, \text{ for }\, n\geq1\, \text{ and } 0 \text{ otherwise,}\\
    &\d\nu(t)=\frac{1}{2\pi}\left\{\Re{\frac{\Gamma'}{\Gamma}\left(\frac{1}{4}+i\frac{t}{2}\right)}-\log{\pi}\right\}\d t-\sum_{\substack{\rho\\ \Re\rho=1/2}}\del_{\ga_\rho},\\
    &\eta=\del_{i/2}+\del_{-i/2}-\sum_{\substack{\rho\\ \Re\rho\neq1/2}}\del_{\ga_\rho},
\end{align*}
and the corresponding generating function given by Theorem \ref{thm:main} takes the form
\begin{align*}
    F(z)=\frac{1}{2\pi}\sum_{n\geq1} \frac{\Lambda(n)}{n^{1/2-iz}}=-\frac{1}{2\pi}\frac{\zeta'}{\zeta}\left(\frac{1}{2}-iz\right), \text{ which holds for } \Im{z}>\frac{1}{2},
\end{align*}
and which extends as a meromorphic function to $\cp^+$ as
\begin{align*}
    F(z)=&\frac{1}{4\pi^2 i}\int_{\r} \frac{1+tz}{t-z}\left\{\Re\frac{\Gamma'}{\Gamma}\left(\frac{1}{4}+i\frac{t}{2}\right)-\log{\pi}\right\}\frac{\d t}{1+t^2}-\frac{4z}{i\pi(1+4z^2)}\\
    &+\frac{1}{2\pi}\sum_{\al} \frac{1}{(\al-1/2)^2-1}\cd\frac{1-(\al-1/2)iz}{(\al-1/2)-iz}.
\end{align*}
In this case $Q(z)\equiv0$ is the zero polynomial because $\zeta(s)$ is real for $s\in\r$.

\section{Preparatory results}
\label{sec:aux_results}
In this section, we develop some preparatory results that we will need to prove Theorems \ref{thm:main} and \ref{thm:converse_main}.
\subsection{The almost periodic class}
\label{section:almost_periodic_class}
We recall that a continuous function $f:\r\ra\cp$ is almost periodic if, for any $\ep>0$, there exists a set of translations $\tau_\ep\subset\r$, which is relatively dense, and such that $\sup_{x\in\r} |f(x+t)-f(x)|<\ep$, for any $t\in\tau_\ep$. We denote by $\A\P(\r)$ the set of almost periodic functions $f:\r\ra\cp$ and by $C(\r)$ the space of bounded continuous functions $f:\r\ra\cp$, endowed with the uniform norm $||f||_{\infty}=\sup_{x\in\r} |f(x)|$. We now present Bochner's criterion for almost periodicity:
\begin{proposition}[Bochner's criterion]
    Let $f:\r\ra\cp$ be a continuous function. Then $f$ is almost periodic if and only if the set $\{f(\cd+h);h\in\r\}$ is pre-compact in $(C(\r),||\cd||_{\infty})$.
\end{proposition}
The proof of this Proposition can be found in \cite[Section 1.2]{AP}. It follows from Bochner's criterion that $\A\P(\r)$ is closed under addition, multiplication, and uniform convergence. In particular, since the function $x\mt e^{2\pi i\la x}$, $x,\la\in\r$, is almost periodic, we conclude that the set $\A\P(\r)$ contains all trigonometric polynomials.

If $f\in\A\P(\r)$, then for each $\la\in\r$, the limit
\begin{align*}
    \e f(\la):=\lim_{T\ra\infty} \frac{1}{2T} \int_{-T}^{T} f(x) e^{-2\pi i\la x}\d x
\end{align*}
does exist. We will call the numbers $\e f(\la)$ the Fourier coefficients of $f$. Moreover, $\spec(f):=\{\la\in\r; \e f(\la)\neq0\}$ is a countable set. The next result concerns how to approximate a function $f\in\A\P(\r)$ using trigonometric polynomials.
\begin{proposition}[Bochner's approximation]
\label{prop:Bochner_approximation}
    For any $f\in\A\P(\r)$, there exists a sequence of trigonometric polynomials $p_n$ such that $||p_n-f||_{\infty}\ra0$. More precisely, there exists a sequence of functions $d_n:\r\ra[0,1]$ such that each $d_n$ has finite support, $d_n\leq d_{n+1}$ and $\lim_{n\ra\infty} d_n(\la)=1$, if $\la\in\spec(f)$, and $\lim_{n\ra\infty} d_n(\la)=0$ otherwise, and if we define 
    \begin{align*}
        p_n(x):=\sum_{\la\in\r} d_n(\la)\e f(\la) e^{2\pi i\la x},
    \end{align*}
    then $p_n\ra f$ in the uniform norm.
\end{proposition} 
For a proof of this proposition, see, for instance, \cite[Section 1.9]{Be}.

A holomorphic map $F:\cp^+\ra\cp$ is almost periodic in the half-plane $\cp^+$ if it is bounded on strips, i.e., $\sup_{\al<\Im{z}<\beta} |F(z)|<\infty$ for any $0<\al<\beta<\infty$, and $F(\cd+iy)\in\A\P(\r)$ for any $y>0$. This is equivalent to the standard definition that states that $F\in\A\P(\cp^+)$ if $F$ is holomorphic and, for any $0<\al<\beta<\infty$ and $\ep>0$, there exists a set of translations $\tau_{\al,\beta,\ep}\subset\r$ that is relatively dense and such that $\sup_{\al<\Im{z}<\beta} |F(z+t)-F(z)|<\ep$, for any $t\in\tau_{\al,\beta,\ep}$. In this case, the limit 
\begin{equation}
\label{eq:def_Fourier_coef_ap_function}
\e F(\la):=\lim_{T\ra\infty} \frac{1}{2T} \int_{-T+iy}^{T+iy} F(z)e^{-2\pi i\la z}dz,
\end{equation}
for $\la\in\r$, exists and is independent of $y>0$, and $\spec(F):=\{\la\in\r; \e F(\la)\neq0\}$. For a proof of these facts, see, for instance, \cite[Section 1.3]{Be} and \cite[Section 3.2]{Be}.

We conclude this subsection with an elementary lemma.
\begin{lemma}
    \label{lemma:bounded_ap_function}
    Let $F\in\A\P(\cp^+)$ and assume it is bounded on the half-plane $\Im{z}>c$, for some $c>0$. Then $\spec(F)\subset[0,\infty)$.
\end{lemma}
\begin{proof}
    Assume $|F(z)|\leq M<\infty$ for $\Im{z}>c$. If $\la<0$, then $|F(z) e^{-2\pi i\la z}|\leq M e^{2\pi i\la y}$ for $\Im{z}>c$. Hence, taking $y>c$ in \eqref{eq:def_Fourier_coef_ap_function} we obtain $|\e F(\la)|\leq M e^{2\pi\la y}$. Then we let $y\ra\infty$ and obtain $\e F(\la)=0$.
\end{proof}

\subsection{The Nevanlinna class}
\label{section:Nevanlinna_class}
In this subsection, we collect the essential facts about the class $\NN_{\leq k}$ of generalized holomorphic Nevanlinna functions. For more detailed information, we refer the reader to  \cite{DL,DLLS,dB,KL,KW,GT14}.

We recall the Herglotz-Nevanlinna representation (\cite[Proposition 2.1]{DL} for the holomorphic scalar case) that states that any holomorphic function $F\in\NN_{\leq k}$ can be represented as
\begin{equation}
\label{eq:Nevanlinna_fact_k}
F(z)=\frac{(z^2+1)^{m}}{2\pi i}\int_{\r} \frac{1+tz}{t-z}\cd\frac{d\nu(t)}{(1+t^2)^{m+1}}+iQ(z),
\end{equation}
for some integer $m\leq k$, a nonnegative Borel measure $\nu$ on $\r$ such that $\ord(\nu)\leq 2(m+1)$ and $Q(z)=a_{2m+1}z^{2m+1}+...+a_1z+a_0$ is a real polynomial of degree at most $2m+1$ such that $a_{2m+1}\leq 0$. The measure $\nu$ is uniquely defined in terms of $F$ by
\begin{equation*}
  \lim_{s\downarrow0} \Re \int_{a+is}^{b+is} \frac{F(z)-iQ(z)}{(z^2+1)^k}\d z = \int_{a}^b \frac{\tfrac12 d\nu(t)}{(1+t^2)^{k+1}},  
\end{equation*}
whenever $a<b$ are points of continuity of $\nu$, and the polynomial $Q(z)$ depends on $F$ and on the exponent $m$. Conversely, any function defined by \eqref{eq:Nevanlinna_fact_k} with $m\leq k$, and with $\nu$ and $Q$ as above, defines a holomorphic function in the class $\NN_{\leq k}$.

Observe that if we define, for some $\rho>0$,
\begin{equation}
\label{eq:Nevanlinna_paramenter}
    F_\rho(z):=\frac{(z^2+\rho^2)^{m}}{2\pi i}\int_{\r} \frac{\rho^2+tz}{t-z}\cd\frac{d\nu(t)}{(\rho^2+t^2)^{m+1}}
\end{equation}
then $i[F_{\rho_1}(z)-F_{\rho_2}(z)]$ extends to a real polynomial of degree at most $2m$. To see that, we apply the identity
\begin{equation}
\label{eq:identity_change_parameter}
    \frac{(z^2+\rho^2)^m(\rho^2+tz)}{(t-z)(\rho^2+t^2)^{m+1}}=\frac{1}{t-z}-\frac{t+z}{\rho^2+t^2}\sum_{j=0}^{m-1} \left(\frac{z^2+\rho^2}{\rho^2+t^2}\right)^j-\frac{t(\rho^2+z^2)^m}{(\rho^2+t^2)^ {m+1}},\,\, \rho>0.
\end{equation}
Then, the common pole and residue cancel, leaving a real polynomial of degree at most $2m$. Moreover, if we define
\begin{align*}
    \tilde{F}_\rho(z):=\frac{(z^2+\rho^2)^{j}}{2\pi i}\int_{\r} \frac{\rho^2+tz}{t-z}\cd\frac{d\nu(t)}{(\rho^2+t^2)^{j+1}},
\end{align*}
for some $j\in\z_{\geq0}$ such that $\ord(\nu)\leq2(j+1)$, then the same identity \eqref{eq:identity_change_parameter} allows us to prove that $i[F_\rho(z)-\tilde{F}_\rho(z)]$ is a real polynomial of degree at most $2\max\{m,j\}$.

If $\mu=\nu+\eta$ is an admissible strongly tempered real measure supported on the strip $\{z;|\Im{z}|<\sigma\}$, with $\ord(\nu),\ord(\eta)\leq 2(m+1)$, then, provided that $\rho>\sigma$, the function \eqref{eq:Nevanlinna_paramenter} is still well defined and takes the form
\begin{equation}
\label{eq:Nevanlinna_F_rho}
    F_\rho(z)= \frac{(z^2+\rho^2)^m}{2\pi i} \int_{\r} \frac{\rho^2+tz}{t-z}\cd\frac{d\nu(t)}{(\rho^2+t^2)^{m+1}}+\frac{(z^2+\rho^2)^m}{2\pi i} \sum_{\ga\in A} \frac{\rho^2+\ga z}{\ga-z}\cd\frac{b(\ga)}{(\rho^2+\ga^2)^{m+1}}.
\end{equation}
The parameter $\rho$ may be changed independently in the integral and the sum, and such a change may add a term of the form $iQ_\rho(z)$ for some polynomial $Q_\rho$ of degree at most $2m$. If we assume, in addition, that $\eta$ is $\r-$symmetric, and since $\eta$ is also real valued, it follows that $Q_\rho$ has real coefficients. Moreover, we can also replace $m$ independently in the integral and the sum by any integers $l,j$ such that $\ord(\nu)\leq2(l+1)$ and $\ord(\eta)\leq2(j+1)$, respectively. This change may add a term $iQ(z)$ for $Q$ a polynomial of degree at most $2\max\{l,j,m\}$. As before, the $\r-$symmetry of $\eta$ guarantees that $Q$ is a real polynomial. All these facts are proved by applying the identity \eqref{eq:identity_change_parameter} with $t=\ga$. Finally, the map $F_\rho$ defined in this way need not belong to $\NN_{\leq k}$. This is why we decompose $F$ into its regular and singular parts in Theorem \ref{thm:main}.

\subsection{Auxiliary functions and the Bridge Lemma} Our main task in this section is to prove a version of the Bridge Lemma \cite[Lemma 5]{GV} for FS-pairs on a strip. For $z,w\in\cp^+$, $x\in \r$ and $k\in \z_{>0}$ we consider the following auxiliary functions
\begin{align*}
G_0(w,z,x)& :=\frac{e^{-2\pi i\w|x|}\1_{x<0}+e^{2\pi iz|x|}\1_{x\geq0}}{z-\w} \\
A_k(x)&:=\underbracket{e^{-2\pi|\cd|}*\cdots * e^{-2\pi|\cd|}(x)}_{k \text{-times}}  \\
G_k(w,z,x)&:=G_0(w,z,\cd)* A_k(x) \quad (k\geq 1).
\end{align*}
These functions satisfy the following properties:
\begin{itemize}
\item [(i)] (Fourier transform) For $k\geq0$, we have 
\begin{equation}
\ft{G_k}(w,z,t)=\frac{1}{2\pi^{k+1}i}\cd\frac{1}{(t-z)(t-\w)}\cd\frac{1}{(1+t^2)^k};
\end{equation}
where the Fourier transform is taken in the last variable.
\item[(ii)] (Anti-symmetry) For $\la\geq0$ we have $G_k(w,z,-\la)=-\ov{G_k(z,w,\la)}$;
\item[(iii)] (Explicit formula) For $z\neq i$, it holds
\begin{align*}
(z-\w)G_k(w,z,\la)&=e^{-2\pi\la}\sum_{j=0}^{k-1} \frac{j!b_{k-1,j}}{(2\pi)^{j+1}(1+i\w)^{j+1}}\sum_{l=0}^{j} \frac{(2\pi\la)^l(1+i\w)^l}{l!}\\
&-e^{-2\pi\la}\sum_{j=0}^{k-1} \frac{j!b_{k-1,j}}{(2\pi)^{j+1}(1+iz)^{j+1}}\sum_{l=0}^{j} \frac{(2\pi\la)^l(1+iz)^l}{l!}\\
&+e^{2\pi i\la z}\cd\frac{1}{\pi^k}\cd\frac{1}{(1+z^2)^k},
\end{align*}
for real coefficients $b_{k-1,j}$ defined by $A_k(x)=e^{-2\pi|x|}p_{k-1}(|x|)$, where $p_{k-1}(x):=\sum_{j=0}^{k-1} b_{k-1,j}x^j$. Such polynomials $p_k$ are encoded in the generating series
\begin{align*}
 \sum_{k\geq 0} q^k \pi^k p_{k}(|x|)  =\frac{e^{(1-\sqrt{1-q})|x|}}{\sqrt{1-q}}
\end{align*}
which converges absolutely for $|q|<1 $ and $x\in \r$.
\end{itemize}
Indeed, item (i) follows from $\ft{e^{-2\pi|\cd|}}(t)=\frac{1}{\pi(1+t^2)}$ combined with the elementary computation $\ft{G_0}(w,z,t)=1/[2\pi i(t-z)(t-\w)]$. Item (ii) follows from the explicit formula (iii), which in turn is proved as follows:
\begin{align*}
G_k(w,z,\la)&=\left(\int_{-\infty}^{-\la}+\int_{-\la}^{0}+\int_{0}^{\infty}\right)\frac{e^{-2\pi i\w |\la+t|}\1_{(-\infty,0)}(\la+t)+e^{2\pi iz|\la+t|}\1_{[0,\infty)}(\la+t)}{z-\w} A_k(t) dt\\
&=:I_1+I_2+I_3.
\end{align*}
After an elementary computation, we arrive at
\begin{align*}
&I_1=\frac{e^{-2\pi\la}}{z-\w}\sum_{j=0}^{k-1} \frac{j!b_{k-1,j}}{(2\pi)^{j+1}(1+i\w)^{j+1}}\sum_{l=0}^{j} \frac{(2\pi\la)^l(1+i\w)^l}{l!},\\
&I_2=\frac{-e^{-2\pi\la}}{z-\w} \sum_{j=0}^{k-1} \frac{j!b_{k-1,j}}{(2\pi)^{j+1}(1+iz)^{j+1}}\sum_{l=0}^{j} \frac{(2\pi\la)^l(1+iz)^l}{l!}+\frac{e^{2\pi i\la z}}{z-\w}\sum_{j=0}^{k-1} \frac{j!b_{k-1,j}}{(2\pi)^{j+1}}\cd\frac{1}{(1+iz)^{j+1}},\\
&\text{ and if } z=i, \text{ then } I_2=\frac{e^{-2\pi\la}}{i-\w}\sum_{j=0}^{k-1} \frac{b_{k-1,j}\la^{j+1}}{j+1}, \text{ and finally }\\
&I_3=\frac{e^{2\pi i\la z}}{z-\w}\sum_{j=0}^{k-1} \frac{j!b_{k-1,j}}{(2\pi)^{j+1}}\cd\frac{1}{(1-iz)^{j+1}}.
\end{align*}
For the last step, given $x\in\r$ and $k\geq1$ set $H_0(z,x):=e^{-2\pi i\ov{z}|x|}\1_{x<0}+e^{2\pi iz|x|}\1_{x\geq0}$ and $H_k(z,x):=H_0(z,\cd)* A_k(x)$. Then for $z=s\in\r$, by using that $A_k(x)=e^{-2\pi|x|}p_{k-1}(|x|)$ we obtain
\begin{align*}
&H_k(s,0)=\int_{\r} e^{2\pi isx}A_k(x)dx=\ft{A_k}(-s)=\frac{1}{\pi^k}\cd\frac{1}{(1+s^2)^k}, \text{ and}\\
&H_k(s,0)=\int_{\r} e^{2\pi isx}A_k(x)dx=\sum_{j=0}^{k-1} \frac{j!b_{k-1,j}}{(2\pi)^{j+1}}\left[\frac{1}{(1+is)^{j+1}}+\frac{1}{(1-is)^{j+1}}\right],
\end{align*}
and by analytic continuation in $s$, the last terms of the above equations coincide for $s\in\cp\backslash\{\pm i\}$. Regarding the generating series, we have
$$
\sum_{k\geq 1} q^{k-1} \ft{A_k}(\xi) = \sum_{k\geq 1} \frac{q^{k-1}}{\pi^k (1+\xi^2)^k} = \frac{1}{\pi(\sqrt{1-q/\pi}^2+\xi^2)},
$$
and by Fourier inversion
$$
\sum_{k\geq 1} q^{k-1} A_k(|y|) = \frac{e^{-2\pi\sqrt{1-q/\pi}|y| }}{\sqrt{1-q/\pi}}.
$$
Finally, we set $x=2\pi y$ and replace $q/\pi$ by $q$.

Let $\U:=\{\zeta\in\cp;|\Im{\zeta}|<1/2\}$. Observe that, for $z,w$ such that $\Im{z},\Im{w}>1$, the map $\ft{G_k}(w,z,\zeta)$ extends to $\U$ as a holomorphic map in the variable $\zeta$.  

For $k$ a non-negative integer, we consider the kernel $S_k$:
\begin{equation*}
S_k(x):=\frac{1}{v_k}\left(\frac{\sin(\pi x)}{\pi x}\right)^{2(k+1)}
\end{equation*}
and we choose $v_k>0$ so that $\ft{S_k}(0)=1$. Its Fourier transform is given by the convolution of $2(k+1)$ indicator functions:
\begin{equation*}
\ft{S_k}(t) = \frac{1}{v_k} \1_{[-1/2,1/2]}*\cdots*\1_{[-1/2,1/2]}(t).
\end{equation*}

We now fix an integer $k\geq0$ and write, for simplicity, $\S:=S_k$. Observe that $\ft{\S}$ is compactly supported with $\supp(\ft{\S})\subset[-k-1,k+1]$ and $\ft{\S}(0)=1$. For $T>0$, we define 
$\S_T(x)=T\S(Tx)$
which is an approximation of identity as $T\ra\infty$, and $\ft{\S_T}(t)=\ft{\S}(t/T)$. Note that, for $\Im{z},\Im{w}>1$, the function $\ft{G_k}(w,z,\cd)*\S_T$, originally defined on $\r$, extends as a holomorphic map of $\zeta$ to the domain $\U$. Indeed, define 
\begin{align*}
    \left(\ft{G_k}(w,z,\cd)*\S_T\right)(\zeta):=\int_{\r} \ft{G_k}(w,z,\zeta-s)\S_T(s)\d s,
\end{align*}
and apply Morera's Theorem.

The next lemma allows us to estimate the asymptotic decay of $\left(\ft{G_k}(w,z,\cd)*\S_T\right)(\zeta)$ as $|\Re{\zeta}|\ra\infty$.

\begin{lemma}
    \label{lemma:asymptotic_decay}
    The following statements hold:
    \begin{itemize}
        \item[i)] Let $f_1,f_2\in L^1(\r)$ such that $|f_1(s)|\leq\frac{C}{(1+s^2)^n}$ and $|f_2(s)|\leq\frac{C}{(1+s^2)^n}$ whenever $|s|>R_1$, for some integer $n\geq1$. Then $|f_1*f_2(s)|\leq\frac{C_1}{(1+s^2)^n}$ for all $|s|\geq2R_1$.
        \item[ii)] Consider the region $\R:=\{z\in\cp^+;|\Re(z)|<1 \text{ and } \Im(z)>1\}$. Then, for any $\zeta\in\U$, $z,w\in\R$, and $T>1$, it holds that
        \begin{equation}
        \label{eq:unif_estimate_lemma}
        \left|\ft{G_k}(w,z,\cd)* \S_T(\zeta)\right|\ll\frac{1}{(1+(\Re{\zeta})^2)^{k+1}}.
        \end{equation}
    \end{itemize}
\end{lemma}
\begin{proof}
    To prove i), observe that if $|x|\geq 2R_1$, then either $|s|>\frac{|x|}{2}\geq R_1$ or $|x-s|\geq\frac{|x|}{2}\geq R_1$. Then
\begin{align*}
    \left|f_1*f_2(x)\right|&\leq \int_{\{s;\,|s|>\frac{|x|}{2}\}} |f_1(s)|\cd|f_2(x-s)|\d s +\int_{\{s;\,|x-s|\geq\frac{|x|}{2}\}} |f_1(s)|\cd|f_2(x-s)|\d s\\
    &\leq\int_{\{s;\,|s|>\frac{|x|}{2}\}} \frac{C}{(1+s^2)^n}\cd|f_2(x-s)|\d s +\int_{\{s;\,|x-s|\geq\frac{|x|}{2}\}} |f_1(s)|\cd\frac{C}{(1+(x-s)^2)^n}\d s\\
    &\leq \frac{4^n C}{(4+x^2)^n}\left(\int_{\r} |f_2(s)|\d s +\int_{\r} |f_1(s)|\d s\right)\\
    &\leq \frac{C_1}{(1+x^2)^n},
\end{align*}
for some constant $C_1>0$.

We now use this result to prove ii). Fix $\zeta=x+iy\in \U$ and $z,w\in\R$. By definition,
\begin{align}
\label{eq:integrals_convolution}
    \left(\ft{G_k}(w,z,\cd)*\S_T\right)(\zeta)&:=\int_{\r} \ft{G_k}(w,z,\zeta-s)\S_T(s)\d s\\
    &=\int_{\{s;|s|\leq1\}} g(x-s)\S_T(s)\d s+\int_{\{s;|s|>1\}} g(x-s)\S_T(s)\d s,
\end{align}
where, for $s\in\r$,
\begin{align*}
    g(s)=g_{w,z,y}(s):=\frac{1}{2\pi^{k+1}i[s+iy-z][s+iy-\ov{w}][1+(s+iy)^2]^k}.
\end{align*}
Observe that, for $|s|\leq1$, we have $|g(x-s)|\leq\frac{C}{(1+x^2)^{k+1}}$, for some constant $C$ that only depends on $k$, as this can be proved by analyzing the cases $|x|<4$ and $|x|\geq4$ separately. Therefore, the absolute value of the first integral in \eqref{eq:integrals_convolution} is bounded by
\begin{align*}
    \frac{C}{(1+x^2)^{k+1}}\int_{\{s;|s|\leq1\}} S_T(s)\d s\leq \frac{C}{(1+x^2)^{k+1}}\int_{\r} S_T(s)\d s=\frac{C}{(1+x^2)^{k+1}}.
\end{align*}
For the case $|s|>1$, let $f_1(s):=g(s)$ and $f_2(s):=\S_T(s)\1_{|s|>1}(s)$, and observe that, for $T>1$, 
\begin{align}
\label{eq:estimates_f1_f2}
    &|f_1(s)|\leq \frac{C_1}{(1+s^2)^{k+1}},  \text{ and } & |f_2(s)|\leq \frac{C_2}{(1+s^2)^{k+1}}
\end{align}
hold for all $s\in\r$, where $C_1,C_2$ are finite constants that depend only on $k$. By item i) above, it follows that $|f_1*f_2(s)|\leq\frac{C_3}{(1+s^2)^{k+1}}$ for $|s|>R$, for some $R>0$. Moreover, since the estimates \eqref{eq:estimates_f1_f2} hold for any $s\in\r$, we conclude that we can choose the constant $C_3$ depending only on $k$. Finally, by \eqref{eq:estimates_f1_f2} we also conclude that $f_1*f_2$ is uniformly bounded on $\zeta\in\U,z,w\in\R$, hence, taking a larger $C_3$, we obtain 
\begin{align*}
    |f_1*f_2(x)|\leq\frac{C_3}{(1+x^2)^{k+1}}
\end{align*}
uniformly on $\zeta\in\U, z,w\in\R$.
\end{proof}

The next lemma is the version of \cite[Bridge Lemma]{GV} for FS-pairs on a strip. Observe that this new version is only proved for the region $\R$ defined on Lemma \ref{lemma:asymptotic_decay} and it only gives pointwise convergence, which turns out to be enough for our purposes.
\begin{lemma}[The Bridge Lemma]
\label{thm:bridge}
If $(\mu=\nu+\eta,a)$ is an FS-pair on a strip such that $\ord(\mu)\leq 2(k+1)$ and $A\subset \{z;|\Im{z}|<1/2\}$, then
\begin{align}
\label{eq:bridge}
\lim_{T\ra\infty} \sum_{|\la|\leq T(k+1)} a(\la)G_k(w,z,\la){\ft{S_k}\left(\frac{\la}{T}\right)}&=\frac{1}{2\pi^{k+1}i}\int_{\r} \frac{1}{(t-z)(t-\w)}\cd\frac{d\nu(t)}{(1+t^2)^k}\\
&+\frac{1}{2\pi^{k+1}i}\sum_{\ga\in A} \frac{1}{(\ga-z)(\ga-\w)}\cd\frac{b(\ga)}{(1+\ga^2)^k}
\end{align}
for $z,w \in \cp^+$ in the region $\R:=\{z\in\cp^+;|\Re{z}|<1 \text{ and } \Im{z}>1\}$.
\end{lemma}
\begin{proof}
Let $z,w\in\R$ and let $\S:=S_k$ and $\S_T(x):=T\S(Tx)$ as defined above. Since the map $t\mt G_k(w,z,t)\ft{S_T}(t)$ is not in $C^{\infty}_c(\r)$, we will need an additional trick: Take a function $\p\in C^\infty_c(\r)$, with $\p\geq0$, $\supp(\p)\subset[-1,1]$ and $\ft{\p}(0)=1$. For $0<\ep<1$, let $\p_\ep(x)=\p(x/\ep)/\ep$. Then, for $w,z\in\R$ we define
\begin{equation}
G_{\ep,T}(x)=\left(G_k(w,z,\cd)\ft{\S_T}\right)*\p_\ep(x)
\end{equation}
which belongs to $C^\infty_c(-T(k+1)-1,T(k+1)+1)$. In particular, the Fourier transform $\ft{G_{\ep,T}}(t)$ extends to an entire function. By analytic continuation, for $\zeta\in\U$ we can write
\begin{equation}
\ft{G_{\ep,T}}(\zeta)=\left(\ft{G_k}(w,z,\cd)* \S_T\right)(\zeta)\ft{\p}(\ep \zeta).
\end{equation}
We also note that $\ft{G_{\ep,T}}(\zeta)\ra\ft{G_k}(w,z,\cd)* \S_T(\zeta)$ pointwise for all $\zeta\in\U$ when $\ep\downarrow0$, and we recall the estimate from Lemma \ref{lemma:asymptotic_decay} (ii)
\begin{equation}
\label{eq:unif_estimate_1}
\left|\ft{G_k}(w,z,\cd)* \S_T(\zeta)\right|\ll\frac{1}{(1+(\Re{\zeta})^2)^{k+1}},
\end{equation}
uniformly on $\zeta\in\U$ and $z,w\in\R$. Since $(\mu=\nu+\eta,a)$ is an FS-pair on a strip with $A\subset\U$, we have
\begin{align}
\sum_{|\la|\leq T(k+1)+1} a(\la)G_{\ep,T}(\la)&=\int_{\r} \left(\ft{G_k}(w,z,\cd)* \S_T\right)(t)\ft{\p}(\ep t) d\nu(t)\\&+\sum_{\ga\in A} b(\ga)\left(\ft{G_k}(w,z,\cd)* \S_T\right)(\ga)\ft{\p}(\ep \ga).
\end{align}
Because $a(\cd)$ is locally summable, $G_{\ep,T}\ra G_k(w,z,\cd)\ft{\S_T}$ uniformly on $\r$, $\ft{G_{\ep,T}}(\zeta)\ra\ft{G_k}(w,z,\cd)* \S_T(\zeta)$ pointwise for $\zeta\in\U$ as $\ep\downarrow0$, $|\ft{\p}(\zeta)|\leq c<\infty$ for $\zeta\in\U$ and inequality \eqref{eq:unif_estimate_1}, using the Dominated Convergence Theorem we obtain that
\begin{align}
\label{eq:intermediate_bridge}
\sum_{|\la|\leq T(k+1)} a(\la)G_k(w,z,\la)\ft{\S}\left(\frac{\la}{T}\right)&=\int_{\r} \ft{G_k}(w,z,\cd)*\S_T(t)d\nu(t)\\
&+\sum_{\ga\in A} b(\ga)\left(\ft{G_k}(w,z,\cd)* \S_T\right)(\ga).
\end{align}
holds for any $w,z\in\R$.

Our next step is to take $T\ra\infty$. Observe that the family of maps $\left\{\zeta\in\U\mt\ft{G_k}(w,z,\zeta)\right\}_{z,w\in \R}$ is uniformly continuous and uniformly bounded, uniformly on $z,w\in\R$. Indeed, each function is Lipschitz with constant independent of $z,w$: for $\zeta_1,\zeta_2\in\U$, we have
\begin{align*}
\left|\ft{G_k}(w,z,\zeta_1)-\ft{G_k}(w,z,\zeta_2)\right|&\ll\left|\frac{\ft{G_0}(w,z,\zeta_1)}{(1+\zeta_1^2)^k}-\frac{\ft{G_0}(w,z,\zeta_2)}{(1+\zeta_2^2)^k}\right|\\
&\leq\left|\frac{\ft{G_0}(w,z,\zeta_1)}{(1+\zeta_1^2)^k}-\frac{\ft{G_0}(w,z,\zeta_1)}{(1+\zeta_2^2)^k}\right|+\left|\frac{\ft{G_0}(w,z,\zeta_1)}{(1+\zeta_2^2)^k}-\frac{\ft{G_0}(w,z,\zeta_2)}{(1+\zeta_2^2)^k}\right|\\
&\leq\left|\ft{G_0}(w,z,\zeta_1)\right|\left|\frac{1}{(1+\zeta_1^2)^k}-\frac{1}{(1+\zeta_2^2)^k}\right|\\
&+\left|\ft{G_0}(w,z,\zeta_1)-\ft{G_0}(w,z,\zeta_2)\right|\frac{1}{(1+\zeta_2^2)^k}\\
&\ll |\zeta_1-\zeta_2|,
\end{align*}
because
\begin{align*}
    \left|\ft{G_0}(w,z,\zeta)\right|\leq 2/\pi\, \text{ and } \frac{1}{\left|1+\zeta^2\right|^k}\leq 4^k \text{ uniformly in } \zeta\in\U,z,w\in\R,
\end{align*}
the map $\zeta\in\U\mt\frac{1}{(1+\zeta^2)^k}$ is Lipschitz, and 
\begin{equation}
|\ft{G_0}(w,z,\zeta_1)-\ft{G_0}(w,z,\zeta_2)|=\frac{1}{2\pi|z-\w|}\left|\frac{1}{\zeta_1-z}-\frac{1}{\zeta_1-\w}-\frac{1}{\zeta_2-z}+\frac{1}{\zeta_2-\w}\right|\leq\frac{2}{\pi}\left|\zeta_1-\zeta_2\right|.
\end{equation}
Since $\ft{G_k}(w,z,\cd)$ is uniformly continuous, we deduce that $\ft{G_k}(w,z,\cd)*\S_T(\zeta)\to \ft{G_k}(w,z,\zeta)$ as $T\ra\infty$, for all $\zeta\in\U$. To see that consider the function defined on the horizontal lines $\r+iy$ and reduce to the usual case of $\r$. Combining \eqref{eq:unif_estimate_1} and \eqref{eq:intermediate_bridge}, and using the Dominated Convergence Theorem, we finally deduce that
\begin{align*}
\lim_{T\ra\infty} \sum_{|\la|\leq T(k+1)} a(\la)G_k(w,z,\la)\ft{\S}\left(\frac{\la}{T}\right)&=\frac{1}{2\pi^{k+1}i}\int_{\r} \frac{1}{(t-z)(t-\w)}\cd\frac{d\nu(t)}{(1+t^2)^k}\\
&+\frac{1}{2\pi^{k+1}i}\sum_{\ga\in A} \frac{1}{(\ga-z)(\ga-\w)}\cd\frac{b(\ga)}{(1+\ga^2)^k},
\end{align*}
for each $z,w\in\R$, as desired. 
\end{proof}
\section{Proof of the main results}
\label{section:proof_main_results}
We now prove Theorem \ref{thm:main}. We begin with a Lemma, which states that it suffices to consider a smaller subclass of FS-pairs.
\begin{lemma}
\label{lemma:simple_proof_thm_main}
It suffices to prove Theorem \ref{thm:main} for FS-pairs on a strip $(\mu=\nu+\eta,a)$ that satisfy $A\subset \{z;|\Im{z}|<1/2\}$ and $\sum_{\la\in\r} |a(\la)|e^{-\pi|\la|}<\infty$.
\end{lemma}
\begin{proof}
Indeed, assume that Theorem \ref{thm:main} is true for that small class of FS-pairs on a strip. Now, let $(\mu=\nu+\eta,a)$ be any FS-pair satisfying the hypothesis of Theorem \ref{thm:main}. Given any $\theta>0$, we can define a new FS-pair on a strip $(\nu_\theta+\eta_\theta,a_\theta)$, where
\begin{align*}
    a_\theta(\la):=& a\left(\frac{\la}{\theta}\right),\\
    \eta_\theta:=&\sum_{\ga\in A_\theta} b_\theta(\ga)\del_\ga,\,\, \text{ where }\,\, A_\theta:=\frac{A}{\theta}=\left\{\frac{z}{\theta};z\in A\right\}\,\, \text{ and }\,\, b_\theta(\ga):=\frac{b(\theta\ga)}{\theta},\\
    \nu_\theta:=&\nu(\theta\cd)\,\,\text{ is the tempered distribution defined by }\,\, \int_{\r} \p(t)\d\nu(\theta t)=\frac{1}{\theta}\int_{\r} \p\left(\frac{t}{\theta}\right)\d\nu(t),\\
    &\text{for }\,\, \p\in C^\infty_c(\r). 
\end{align*}
Since $a(\cd)$ has finite exponential growth and $A\subset\{z;|\Im{z}|<\si_\mu\}$ for some $\si_\mu>0$, by choosing $\theta>0$ large enough we obtain $\sum_{\la>0} |a_\theta(\la)|e^{-\pi|\la|}<\infty$ and $A_\theta\subset \{z;|\Im{z}|<1/2\}$. We now apply Theorem \ref{thm:main} for $(\nu_\theta+\eta_\theta,a_\theta)$ and obtain the map
\begin{equation}
\label{eq:def_F_theta}
    F_\theta(z):=\frac{a_\theta(0)}{2}+\sum_{\la>0} a_\theta(\la)e^{2\pi i\la z},
\end{equation}
which is defined for $\Im{z}>1/2$ and converges uniformly in that region. Moreover, $F_\theta$ extends to $\cp^+$ as a meromorphic function given by
\begin{align*}
    F_\theta(z)=\frac{(z^2+1)^k}{2\pi i} \int_{\r} \frac{1+tz}{t-z}\cd\frac{d\nu_\theta(t)}{(1+t^2)^{k+1}}+\frac{(z^2+1)^k}{2\pi i} \sum_{\ga\in A_\theta} \frac{1+\ga z}{\ga-z}\cd\frac{b_\theta(\ga)}{(1+\ga^2)^{k+1}}+iQ_\theta(z),
\end{align*}
for any $z\in\cp^+\backslash A_\theta$, where $Q_\theta(z)$ is a real polynomial of degree at most $2k$. In particular, it follows from \eqref{eq:def_F_theta} that $F_\theta$ is bounded in the half-plane $\Im{z}>1/2$ and that $F_\theta$ is the uniform limit of trigonometric polynomials on each horizontal line in this half-plane, hence $F_\theta(\cd+iy)\in\A\P(\r)$ for each $y>1/2$. Therefore, $F_\theta(\cd+i/2)\in\A\P(\cp^+)$ and the map $\la\mt\e F_\theta(\la)$ has finite exponential growth, where $\e F_\theta(\la)=a_\theta(\la)$ if $\la>0$, $\e F_\theta(0)=a_\theta(0)/2$, and zero otherwise.

Changing the parameter 1 to $\theta$ in the integral term, we can write 
\begin{align*}
    F_\theta(z)=\frac{(z^2+\theta^{-2})^k}{2\pi i} \int_{\r} \frac{\theta^{-2}+tz}{t-z}\cd\frac{d\nu_\theta(t)}{(\theta^{-2}+t^2)^{k+1}}+\frac{(z^2+1)^k}{2\pi i} \sum_{\ga\in A_\theta} \frac{1+\ga z}{\ga-z}\cd\frac{b_\theta(\ga)}{(1+\ga^2)^{k+1}}+i\widetilde{Q_\theta}(z),
\end{align*}
for a suitable real polynomial $\widetilde{Q_\theta}(z)$ of degree at most $2k$.
We now define $F(z):=F_\theta\left(\frac{z}{\theta}\right)$, and observe that, for $\im{z}>\theta/2$
\begin{align*}
    F(z)=\frac{a(0)}{2}+\sum_{\la>0} a(\la)e^{2\pi i\la z},
\end{align*}
with uniform convergence on that half-plane, and also
\begin{equation}
\label{eq:lemma_simple_proof_main_thm_F_expansion}
    F(z)=\frac{(z^2+1)^k}{2\pi i} \int_{\r} \frac{1+tz}{t-z}\cd\frac{d\nu(t)}{(1+t^2)^{k+1}}+\frac{(z^2+\theta^2)^k}{2\pi i} \sum_{\ga\in A} \frac{\theta^2+\ga z}{\ga-z}\cd\frac{b(\ga)}{(\theta^2+\ga^2)^{k+1}}+iQ(z),
\end{equation}
where $Q(z):=\widetilde{Q_\theta}\left(\frac{z}{\theta}\right)$ is a real polynomial of degree $\leq 2k$.

Finally, the map $z\mt F(z+i\theta/2)\in\A\P(\cp^+)$, $\la\mt\e F(\la)$ has finite exponential growth, where $\e F(\la)=a(\la)$ if $\la>0$, $\e F(0)=a(0)/2$, and zero otherwise. Moreover, we can replace the parameter $\theta$ in \eqref{eq:lemma_simple_proof_main_thm_F_expansion} by any $\rho>\si_\mu$, by replacing $Q(z)$ by a suitable real polynomial $Q_\rho(z)$ of degree at most $2k$. The proof is complete.
\end{proof}
We are now ready to prove Theorem \ref{thm:main}.
\begin{proof}[Proof of Theorem \ref{thm:main}.]
We are given an FS-pair $(\mu=\nu+\eta,a)$ on a horizontal strip. By Lemma \ref{lemma:simple_proof_thm_main}, we can assume, without loss of generality, that $A\subset\{z;|\Im{z}|<1/2\}$ and that $\sum_{\la\in\r} |a(\la)|e^{-\pi|\la|}<\infty$. 

We apply the Bridge Lemma \ref{thm:bridge} and multiply both sides by $(z-\w)$ to obtain
\begin{align}
\label{eq:thm1_bridge}
    \lim_{T\ra\infty} \sum_{|\la|\leq T(k+1)}a(\la)(z-\w)G_k(w,z,\la)\ft{S_k}\left(\frac{\la}{T}\right)&=\frac{1}{2\pi^{k+1}i} \int_{\r} \frac{z-\w}{(t-z)(t-\w)}\cd \frac{\d\nu(t)}{(1+t^2)^k}\\
    &+\frac{1}{2\pi^{k+1}i} \sum_{\ga\in A} \frac{z-\w}{(\ga-z)(\ga-\w)}\cd \frac{b(\ga)}{(1+\ga^2)^k},
\end{align}
which holds for each $z,w$ in the region $\R=\{\zeta\in\cp^+;|\Re{\zeta}|<1\, \text{ and }\, \Im{\zeta}>1\}$. Since the measure $\eta$ is $\r$-symmetric, we can write the right-hand side of this equation as $H(z)+\ov{H(w)}$, where
\begin{align*}
    H(z):=\frac{1}{2\pi^{k+1}i} \int_{\r} \frac{1+tz}{t-z}\cd \frac{\d\nu(t)}{(1+t^2)^{k+1}}+\frac{1}{2\pi^{k+1}i} \sum_{\ga\in A} \frac{1+\ga z}{\ga-z}\cd \frac{b(\ga)}{(1+\ga^2)^{k+1}},
\end{align*}
which defines a meromorphic map in $\cp^+$, with simple poles at $\ga\in A\cap\cp^+$.

Since $\sum_{\la\in\r} |a(\la)|e^{-\pi|\la|}<\infty$, it follows that
\begin{equation}
\label{eq:thm1_def_of_F}
    F(z):=\frac{a(0)}{2}+\sum_{\la>0} a(\la)e^{2\pi i\la z}
\end{equation}
defines a holomorphic map in the half-plane $\Im{z}>1/2$, because the series converges uniformly and absolutely in that region. We now apply to the left-hand side of \eqref{eq:thm1_bridge} the properties and explicit formula (i), (ii) and (iii) for the map $G_k$ from Section \ref{sec:aux_results}. By splitting the limit into the sum and isolating the $z-$terms from the $w-$terms, for $z,w\in\R$, we obtain
\begin{align}
\label{eq:bridge_lemma_proof_thm_main}
&F(z)\cd\frac{1}{\pi^k}\cd\frac{1}{(1+z^2)^k}+\ov{F(w)}\cd\ov{\frac{1}{\pi^k}\cd\frac{1}{(1+w^2)^k}}\\
&+\sum_{j=0}^{k-1} \frac{j!b_{k-1,j}}{(2\pi)^{j+1}} \sum_{l=0}^{j} \frac{1}{(1-iz)^{j+1-l}}\ov{\left(\frac{F^{(l)}(i)}{l!i^l}\right)}-\sum_{j=0}^{k-1} \frac{j!b_{k-1,j}}{(2\pi)^{j+1}} \sum_{l=0}^{j} \frac{1}{(1+iz)^{j+1-l}}\frac{F^{(l)}(i)}{l!i^l} \\
&+\sum_{j=0}^{k-1} \frac{j!b_{k-1,j}}{(2\pi)^{j+1}} \sum_{l=0}^{j} \frac{1}{(1+i\w)^{j+1-l}}\frac{F^{(l)}(i)}{l!i^l} -\sum_{j=0}^{k-1} \frac{j!b_{k-1,j}}{(2\pi)^{j+1}} \sum_{l=0}^{j} \frac{1}{(1-i\w)^{j+1-l}}\ov{\left(\frac{F^{(l)}(i)}{l!i^l}\right)}.
\end{align}
Then, \eqref{eq:bridge_lemma_proof_thm_main} can be written as $R(z)+\ov{R(w)}$, where
\begin{align}
    R(z):=&\frac{F(z)}{\pi^k(1+z^2)^k}+\sum_{j=0}^{k-1} \frac{j!b_{k-1,j}}{(2\pi)^{j+1}}\sum_{l=0}^{j} \frac{1}{(1-iz)^{j+1-l}}\ov{\left(\frac{F^{(l)}(i)}{l!i^l}\right)}\\
    &-\sum_{j=0}^{k-1} \frac{j!b_{k-1,j}}{(2\pi)^{j+1}}\sum_{l=0}^{j} \frac{1}{(1+iz)^{j+1-l}}\frac{F^{(l)}(i)}{l!i^l}
\end{align}
is holomorphic in the half-plane $\Im{z}>1$. Hence, $R(z)+\ov{R(w)}=H(z)+\ov{H(w)}$ for $z,w\in\R$. Taking $z=w$, we conclude that $\Re{R(z)}=\Re{H(z)}$ there. The Cauchy--Riemann equations then imply that $R(z)=ih+H(z)$ for some $h\in\r$, and the identity extends to $\Im{z}>1$ by analytic continuation. Multiplying both sides by $\pi^k(z^2+1)^k$ and rearranging terms, we obtain
\begin{equation}
\label{eq:thm1_F_extension}
    F(z)=\frac{(z^2+1)^k}{2\pi i} \int_{\r} \frac{1+tz}{t-z}\cd \frac{\d\nu(t)}{(1+t^2)^{k+1}}+\frac{(z^2+1)^k}{2\pi i} \sum_{\ga\in A} \frac{1+\ga z}{\ga-z}\cd \frac{b(\ga)}{(1+\ga^2)^{k+1}}+iQ(z),
\end{equation}
which now holds for $\Im{z}>1/2$, where $Q(z)$ is a real polynomial of degree $\leq 2k$, defined by
\begin{align*}
    &Q(z)=h\pi^k(z^2+1)^k-i\pi^k\sum_{l=0}^{k-1} \left[Q_l(z)-\ov{Q_l{(\ov{z}})}\right],\\
    &Q_l(z)=(z^2+1)^k\frac{F^{(l)}(i)}{l!i^l}\sum_{j=l}^{k-1} \frac{j!b_{k-1,j}}{(2\pi)^{j+1}}\cd\frac{1}{(1+iz)^{j+1-l}}.
\end{align*}
Equation \eqref{eq:thm1_F_extension} also provides a meromorphic extension of $F(z)$ to $\cp^+$, with simple poles at $\ga\in A\cap\cp^+$ and respective residue $-\frac{b(\ga)}{2\pi i}$. Moreover, we can change the parameter 1 by any $\rho>1/2$ and write
\begin{equation*}
    F(z)=\frac{(z^2+1)^k}{2\pi i} \int_{\r} \frac{1+tz}{t-z}\cd \frac{\d\nu(t)}{(1+t^2)^{k+1}}+\frac{(z^2+\rho^2)^k}{2\pi i} \sum_{\ga\in A} \frac{\rho^2+\ga z}{\ga-z}\cd \frac{b(\ga)}{(\rho^2+\ga^2)^{k+1}}+iQ_\rho(z),
\end{equation*}
for a suitable real polynomial $Q_\rho(z)$ of degree at most $2k$.

From the definition of $F(z)$ in \eqref{eq:thm1_def_of_F}, we conclude that $F$ is the uniform limit of trigonometric polynomials on each horizontal line of the half-plane $\Im{z}>1/2$. Hence $F(\cd+iy)\in\A\P(\r)$ for each $y>1/2$. Since $F(z)$ is holomorphic and bounded in that half-plane we conclude that $F(\cd+i/2)\in\A\P(\cp^+)$. Moreover, it has Fourier coefficients $\e F(\la)=a(\la)$, for $\la>0$, $\e F(0)=\frac{a(0)}{2}$, which has finite exponential growth. The proof is complete.  
\end{proof}
Now we present the proof of Theorem \ref{thm:converse_main}. We begin with an elementary lemma.
\begin{lemma}
    \label{lemma:Cauchy_integral}
    Let $\p\in C^\infty_c(\r)$, $k\geq1$ and $w\in\cp\backslash\r$. Then
    \begin{align}
        \frac{1}{2\pi i} \int_{\r} \frac{t^k\ft{\p}(t)}{t-w}\d t=\begin{cases}
            \displaystyle\frac{1}{(2\pi i)^k} \int_{-\infty}^{0} \p^{(k)}(x)e^{-2\pi iwx}\d x,\qquad &\Im{w}>0,\\
            \displaystyle\frac{-1}{(2\pi i)^k} \int_{0}^{\infty} \p^{(k)}(x)e^{-2\pi iwx}\d x,\qquad &\Im{w}<0.
        \end{cases}
    \end{align}
\end{lemma}
\begin{proof}
If $\Im{w}>0$, let $g_w(x):=e^{2\pi iwx}\1_{(0,\infty)}(x)$. Then $g_w\in L^1(\r)$, and $\ft{g_w}(t)=1/[2\pi i(t-w)]$ for $t\in\r$. Then
\begin{align*}
    \frac{1}{2\pi i} \int_{\r} \frac{t^k\ft{\p}(t)}{t-w}\d t&=\frac{1}{(2\pi i)^k} \int_{\r} \ft{\p^{(k)}}(t)\ft{g_w}(t)\d t=\frac{1}{(2\pi i)^k} \int_{\r} \ft{\ft{\p^{(k)}}}(x)g_w(x)\d x\\
    &=\frac{1}{(2\pi i)^k} \int_{0}^{\infty} \p^{(k)}(-x)e^{2\pi iwx}\d x=\frac{1}{(2\pi i)^k} \int_{-\infty}^{0} \p^{(k)}(x)e^{-2\pi iwx}\d x,
\end{align*}
as desired. The case $\Im{w}<0$ is analogous.
\end{proof}

\begin{proof}[Proof of Theorem \ref{thm:converse_main}.]
Define the measure $\eta:=\sum_{\ga\in A} b(\ga)\del_\ga$, where $A:=\{\ga_n\}_{n\geq0}\cup\{\ov{\ga_n}\}_{n\geq0}$ and $b(\ov{\ga_n}):=b(\ga_n)$, which is $\r$-symmetric and $\ord(\eta)\leq2(m+1)$. Since $F-S\in\NN_{\leq j}-\NN_{\leq j}$, using the Herglotz-Nevanlinna representation with exponent $k:=\max\{m,j\}$, we can write for $z\in\cp^+$ and a suitable real polynomial $Q(z)$ of degree $\leq 2k+1$
\begin{equation}
\label{eq:thm2_F_representation}
    F(z)= \frac{(z^2+1)^k}{2\pi i} \int_{\r} \frac{1+tz}{t-z}\cd\frac{d\nu(t)}{(1+t^2)^{k+1}}+\frac{(z^2+\rho^2)^k}{2\pi i} \sum_{\ga\in A} \frac{\rho^2+\ga z}{\ga-z}\cd\frac{b(\ga)}{(\rho^2+\ga^2)^{k+1}}+iQ(z),
\end{equation}
for some strongly tempered, real-signed measure $\nu$ such that $\ord{(\nu)}\leq 2(j+1)\leq2(k+1)$. Since $F(z)$ is almost periodic in the half-plane $\Im{z}>c$, it is holomorphic there, and hence we deduce that $A\subset\{z;|\Im{z}|\leq c\}$.

Since $F(\cd+ic)\in\A\P(\cp^+)$ then the limit
\begin{align*}
\e F(\la):=\lim_{T\ra\infty} \frac{1}{2T}\int_{-T+iy}^{T+iy} F(z)e^{-2\pi i\la z}dz
\end{align*}
does exist for any $\la\in\r$ and is independent of $y>c$. By hypothesis, $\la\mt\e F(\la)$ is locally summable.

We now claim that the coefficients $\e F(\la)$ vanish for $\la<0$. By Lemma \ref{lemma:bounded_ap_function}, it suffices to show that $F$ is bounded on the half-plane $\Im{z}>2c$. We apply a Phragm\'{e}n-Lindel\"{o}f type argument (see \cite[Theorem 1]{dB}): it is enough to show that
\begin{itemize}
    \item[i)] $F(x+2ic)$ is bounded for $x\in\r$;
    \item[ii)] \begin{align*}
        \liminf_{r\ra\infty} \frac{1}{r} \int_{0}^{\pi} \log^+|F(re^{i\theta}+2ic)|\sin{\theta}\d\theta=0,
    \end{align*}
    where $\log^+(x):=\max\{\log{x},0\}$.
\end{itemize}
Indeed, the item $i)$ is a direct consequence of $F(\cd+2ic)\in\A\P(\r)$. The item $ii)$ follows from the representation \eqref{eq:thm2_F_representation} plus the fact that $\left|\frac{\xi^2+\ga z}{\ga-z}\right|\leq\frac{\xi^2+2|z|^2}{c}$ for any $\xi>0$ and $\ga,z$ such that $|\Im{\ga}|\leq c$ and $\Im{z}\geq2c$. Therefore, $|F(re^{i\theta}+2ic)|=O(r^{2k+2})$, and item $ii)$ follows.

All that remains to be proved is the summation property $\sum_{\la\in\r} a(\la)\p(\la)=\int_{\r} \ft{\p}(t)\d\nu(t)+\sum_{\ga\in A} b(\ga)\ft{\p}(\ga)$ for any test function $\p$. The new feature of this proof is the presence of the discrete off-real measure $\eta$, which did not appear in \cite[Theorem 2]{GV}.

Fix a test function $\p\in C^\infty_c(\r)$, and let $L>0$ be such that $\supp{\p}\subset [-L,L]$. Let $H(z):=F(z)-iQ(z)=G(z)+S(z)$, where
\begin{align*}
    &G(z):=\frac{(z^2+1)^k}{2\pi i} \int_{\r} \frac{1+tz}{t-z}\cd\frac{d\nu(t)}{(1+t^2)^{k+1}}, \text{ and we recall that }\\
    &S(z)=\frac{(z^2+\rho^2)^k}{2\pi i} \sum_{\ga\in A} \frac{\rho^2+\ga z}{\ga-z}\cd\frac{b(\ga)}{(\rho^2+\ga^2)^{k+1}}.
\end{align*}
For $z$ such that $\Im{z}>2c$ and $s\in\r$, we consider the integral
\begin{equation}
\label{eq:H_indentity}
    \int_{\r} \left[H(z+s)+\ov{H(-\ov{z}+s)}\right]\ft{\p}(s)\d s
\end{equation}
which is going to be evaluated in two ways: one using \eqref{eq:thm2_F_representation}, the other using the almost periodicity of $F(z)$ in the half-plane $\Im{z}>c$. We split the computations into three steps:
\subsection*{Step 1: The terms associated with the measure $\eta$:} Let $z=x+iy$, with $y>2c$. Then from the definition of $S(z)$ and applying Fubini, we can write
\begin{align}
\label{eq:integral_S}
    \int_{\r} \left[S(z+s)+\ov{S(-\ov{z}+s)}\right]\ft{\p}(s)\d s=&-\sum_{\ga\in A} \frac{b(\ga)}{(\rho^2+\ga^2)^{k+1}}\cd\frac{1}{2\pi i} \int_{\r} \frac{g_{z,\ga}(s)\ft{\p}(s)}{s-(\ga-z)}\d s\\
    &+\sum_{\ga\in A} \frac{\ov{b(\ga)}}{(\rho^2+\ov{\ga}^2)^{k+1}}\cd\frac{1}{2\pi i} \int_{\r} \frac{h_{z,\ga}(s)\ft{\p}(s)}{s-(\ov{\ga}+z)}\d s,
\end{align}
where
\begin{align*}
    &g_{z,\ga}(s)=\left((z+s)^2+\rho^2\right)^k\left(\rho^2+\ga(z+s)\right)=\sum_{j=0}^{2k+1} c_j(z,\ga)s^j, \text{ and }\\
    &h_{z,\ga}(s)=\left((-z+s)^2+\rho^2\right)^k\left(\rho^2+\ov{\ga}(-z+s)\right)=\sum_{j=0}^{2k+1} d_j(z,\ov{\ga})s^j
\end{align*}
are polynomials in $s$ of degree at most $2k+1$, with $c_j(z,\ga)$ and $d_j(z,\ga)$ polynomials in $z$ and $\ga$, with degree in $\ga$ at most 1. Let $S_1(z)$ and $S_2(z)$ denote the first and second terms on the right-hand side of \eqref{eq:integral_S}, respectively. Since $\Im{(\ga-z)}<0$ and $\Im{(\ov{\ga}+z)}>0$, we deduce that, from Lemma \eqref{lemma:Cauchy_integral}
\begin{align*}
    S_1(z)&=\sum_{\ga\in A} \frac{b(\ga)}{(\rho^2+\ga^2)^{k+1}}\sum_{j=0}^{2k+1} \frac{c_j(z,\ga)}{(2\pi i)^j} \int_{0}^{L} \p^{(j)}(x)e^{-2\pi i(\ga-z)x}\d x, \text{ and}\\
    S_2(z)&=\sum_{\ga\in A} \frac{\ov{b(\ga)}}{(\rho^2+\ov{\ga}^2)^{k+1}}\sum_{j=0}^{2k+1} \frac{d_j(z,\ov{\ga})}{(2\pi i)^j} \int_{-L}^{0} \p^{(j)}(x)e^{-2\pi i(\ov{\ga}+z)x}\d x.
\end{align*}
We now claim that $S_1$ and $S_2$ are entire functions. Indeed, we can write $c_j(z,\ga)=c_j^\flat(z)+\ga c_j^\sharp(z)$, with $c_j^\flat(z)$ and $c_j^\sharp(z)$ polynomials in $z$. Then integration by parts allows us to write
\begin{align*}
    S_1(z)=\sum_{\ga\in A} \frac{b(\ga)}{(\rho^2+\ga^2)^{k+1}}\sum_{j=0}^{2k+1} &\frac{c_j^\flat(z)+z c_j^\sharp(z)}{(2\pi i)^j} \int_{0}^{L} \p^{(j)}(x)e^{-2\pi i(\ga-z)x}\d x\\
    &+\frac{c_j^\sharp(z)}{(2\pi i)^{j+1}}\left[\p^{(j)}(0)+ \int_{0}^{L} \p^{(j+1)}(x)e^{-2\pi i(\ga-z)x}\d x\right].
\end{align*}
Then, for $z\in C$, with $C\subset\cp$ a compact set, each term in the inner sum is $\ll_{C,\p} 1$, uniformly in $\ga\in A$. Since $\sum_{\ga\in A} |b(\ga)|(\rho^2+\ga^2)^{-k-1}<\infty$, an application of Morera's theorem gives that $S_1$ is entire. For the same reasoning, $S_2$ is also an entire function.

We now compute $\lim_{y\downarrow0} S_1(iy)+S_2(iy)$. Indeed, $g_{0,\ga}(s)=(s^2+\rho^2)^k(\rho^2+\ga s)=\sum_{j=0}^{2k+1} c_j(0,\ga)s^j$ and $h_{0,\ga}(s)=(s^2+\rho^2)^k(\rho^2+\ov{\ga}s)=\sum_{j=0}^{2k+1} d_j(0,\ov{\ga})s^j$, from which we conclude that $c_j(0,\ga)=d_j(0,\ga)$. Since $b(\ga)\in\r$ and $b(\ov{\ga})=b(\ga)$, where the latter follows from $\r-$symmetry, and since $\int_{-L}^{L} \p^{(j)}(x)e^{-2\pi i\ga x}\d x=\ft{\p^{(j)}}(\ga)=(2\pi i\ga)^j\ft{\p}(\ga)$, we arrive at
\begin{align}
    \lim_{y\downarrow 0} [S_1(iy)+S_2(iy)]=\sum_{\ga\in A} b(\ga)\ft{\p}(\ga).
\end{align}
\subsection*{Step 2: The terms associated with the measure $\nu$:} For $z\in\cp^+$ and $s\in\r$, we have
\begin{align*}
G(z+s)+\ov{G(-\ov{z}+s)}&= \int_{\r} P_z(t-s)(1+s^2)^k(1+t^2)\frac{d\nu(t)}{(1+t^2)^{k+1}}\\
&+2k \int_{\r} P_z(t-s)s(1+s^2)^{k-1}(st-s^2+t^2s^2-ts^3)\frac{d\nu(t)}{(1+t^2)^{k+1}}\\
&+\frac{1}{2} \int_{\r} P_z(t-s)h(z,s,t)\frac{d\nu(t)}{(1+t^2)^{k+1}},
\end{align*}
where $h(z,s,t)$ is a real polynomial in variables $z,s,t$ such that there is no constant term in $z$, the degree in $t$ is at most 2, and $P_z(t):=\frac{z}{i\pi (t^2-z^2)}$ is the Poisson kernel. Hence
\begin{equation}
    \label{eq:contribution_mu}
    \int_{\r} \left[G(z+s)+\ov{G(-\ov{z}+s)}\right]\ft{\p}(s)\d s=G_1(z)+G_2(z), 
\end{equation}
where
\begin{align*}
    G_1(z):=&\int_{\r} \int_{\r} P_z(t-s)(1+s^2)^k\ft{\p}(s)\d s(1+t^2)\frac{\d\nu(t)}{(1+t^2)^{k+1}}, \text{ and }\\
    G_2(z):=&2k \int_{\r} \int_{\r} P_z(t-s)s(1+s^2)^{k-1}(st-s^2+t^2s^2-ts^3)\ft{\varphi}(s)ds\frac{d\nu(t)}{(1+t^2)^{k+1}}\\
    &+\frac{1}{2} \int_{\r} \int_{\r} P_z(t-s)h(z,s,t)\ft{\varphi}(s)ds\frac{d\nu(t)}{(1+t^2)^{k+1}}.
\end{align*}
Note that $G_1$ and $G_2$ are holomorphic in the region  $\W:=\{z\in\cp;|\Re{z}|<1\, \text{ and }\, 0<\Im{z}<3c\}$, and $\lim_{y\ra 0+} G_2(iy)=0$, because $P_{iy}(t)$ is an approximation of identity as $y\ra0^+$.

\subsection*{Step 3: Evaluating the integral using its Fourier series} We now evaluate the integral \eqref{eq:H_indentity} using the almost periodicity of $F$. Fix a $z=x+i(c+\eta)$, $\eta>0$, and we apply Proposition \ref{prop:Bochner_approximation} to the function $F(\cd+i(c+\eta))\in\A\P(\r)$: there exists a sequence of functions $d_n:\r\ra[0,1]$ each one with finite support and with $\lim_{n\ra\infty} d_n(\la)=1$ if $\la\in\spec(F)$ and is zero otherwise, and such that the sequence of trigonometric polynomials
\begin{align*}
    p_n(x):=\sum_{\la\in\r} d_n(\la)\e F(\la)e^{-2\pi\la(c+\eta)}e^{2\pi i\la s}
\end{align*}
converges to $F(\cd+i(c+\eta))$ in the uniform norm. Writing $Q(z+s)=\sum_{l,j=0}^{2k+1} \theta_{l,j} z^l s^j$, $\theta_{l,j}\in\r$, we obtain
\begin{align*}
\int_{\r} \left[p_n(x+s)-iQ(z+s)\right]\ft{\varphi}(s)ds&=\sum_{0\leq\la\leq L} d_n(\la) \e F(\la)e^{2\pi i\la z}\varphi(\la)\\
&-i\sum_{l,j=0}^{2k+1} \theta_{l,j} z^l\frac{\varphi^{(j)}(0)}{(2\pi i)^j}.
\end{align*}
Letting $n\ra\infty$ and using the dominated convergence theorem, plus the fact that the function $\la\mt\e F(\la)$ is locally summable, we arrive at
\begin{equation*}
\int_{\r} H(z+s)\ft{\varphi}(s)ds=\sum_{0\leq\la\leq L} \e F(\la)e^{2\pi i\la z}\varphi(\la)-i\sum_{l,j=0}^{2k+1} \theta_{l,j} z^l\frac{\varphi^{(j)}(0)}{(2\pi i)^j}.
\end{equation*}
Applying the same argument to the integral with $\ov{H(-\ov{z}+s)}$, we obtain for $\Im{z}>c$
\begin{align}
\label{eq:contribution_a}
    \int_{\r} \left[H(z+s)+\ov{H(-\ov{z}+s)}\right]\ft{\p}(s)\d s=\sum_{0\leq\la\leq L} \e F(\la)e^{2\pi i\la z}\p(\la)-i\sum_{l,j=0}^{2k+1} \theta_{l,j} z^l\frac{\p^{(j)}(0)}{(2\pi i)^j}\\
    +\sum_{-L\leq\la\leq 0} \ov{\e F(-\la)}e^{-2\pi i\la z}\p(\la)+i\sum_{l,j=0}^{2k+1} \theta_{l,j} (-z)^l\frac{\p^{(j)}(0)}{(2\pi i)^j}=:T(z).
\end{align}
Since $\p\in L^\infty(\r)$ and $\la\mt\e F(\la)$ is locally summable,  the map $T(z)$ defined above has a holomorphic extension to $\W$.
\subsection*{The final step: taking the limit $z=iy\ra0$:} From \eqref{eq:H_indentity}, \eqref{eq:integral_S}, \eqref{eq:contribution_mu} and \eqref{eq:contribution_a}, we have
\begin{align*}
    T(z)=G_1(z)+G_2(z)+S_1(z)+S_2(z)
\end{align*}
holds for $\Im{z}>2c$. By analytic continuation, the equality is also true in the region $\W$. We now take $z=iy$ and send $y\downarrow0$: We have already proved that
\begin{align*}
    S_1(iy)+S_2(iy)\ra \sum_{\ga\in A} b(\ga)\ft{\p}(\ga).
\end{align*}
Using the fact that the Poisson kernel $P_{iy}(t)$ is an approximation to the identity as $y\downarrow0$, we obtain
\begin{align*}
    G_1(iy)+G_2(iy)\ra \int_{\r} \ft{\p}(t) \d\nu(t). 
\end{align*}
Finally, since the function $\la\ra\e F(\la)$ is locally summable, the dominated convergence theorem yields
\begin{align*}
    T(iy)&\ra \sum_{0\leq\la\leq L} \e F(\la)\p(\la)-i\sum_{j=0}^{2k+1} \theta_{0,j} \frac{\p^{(j)}(0)}{(2\pi i)^j}+\sum_{-L\leq\la\leq 0} \ov{\e F(-\la)}\p(\la)+i\sum_{j=0}^{2k+1} \theta_{0,j} \frac{\p^{(j)}(0)}{(2\pi i)^j}\\
    &=\sum_{\la\in\r} a(\la)\p(\la).
\end{align*}
Therefore,
\begin{align*}
    \sum_{\la\in\r} a(\la)\p(\la)=\int_{\r} \ft{\p}(t) \d\nu(t)+\sum_{\ga\in A} b(\ga)\ft{\p}(\ga),
\end{align*}
which concludes the proof.
\end{proof}
Finally, we present the proof of Corollary \ref{cor:bijection}.
\begin{proof}[Proof of Corollary \ref{cor:bijection}]
Indeed, if $(\nu_1+\eta_1,a_1)$ and $(\nu_2+\eta_2,a_2)$ give rise to the same function $F$, then comparing its Fourier coefficients it follows that $a_1=a_2$. The measures $\eta_1$ and $\eta_2$ are uniquely determined by the poles and residues of $F$, hence $\eta_1=\eta_2$. Finally, the uniqueness of the measure in the Herglotz-Nevanlinna factorization guarantees that $\nu_1=\nu_2$. This proves injectivity.

To show that the correspondence is surjective, let $G\in\G$. Then $G\in\A\P(\cp^++ic)$ for some $c>0$ and $\sum_{\la\in\r} |\e G(\la)| e^{-2\pi|\la|c_1}<\infty$ for some $c_1>0$. Since we also have that $G\in\G$, then following the proof of Theorem \ref{thm:converse_main} we conclude that $G$ is bounded on some half-plane, hence $\e G(\la)=0$ if $\la<0$. By Bochner's approximation (Proposition \ref{prop:Bochner_approximation}) and the dominated convergence theorem, we deduce that
\begin{equation*}
    G(z)=\e G(0)+\sum_{\la>0} \e G(\la) e^{2\pi i\la z},\quad \Im{z}>c_1.
\end{equation*}
We apply Theorem \ref{thm:converse_main} to $G$ and we obtain a pair $(\mu,a)\in\F\S$, where $a(\la)=\e G(\la)$ if $\la>0$, and $a(0)=2 \e G(0)$. Then the map $\F\S\ra\G$ sends $(\mu,a)$ to the function
\begin{equation*}
    F(z):=\frac{a(0)}{2}+\sum_{\la>0} a(\la)e^{2\pi i\la z},
\end{equation*}
which converges absolutely and uniformly in the region $\Im{z}>c_1$. We conclude that $F(z)=G(z)$ for $\Im{z}>c_1$. Therefore, $F=G$ on $\cp^+$ by analytic continuation.
\end{proof}

\subsection*{Acknowledgements.} This work was partially carried out during the author's visit to The University of Texas at Austin, supported by a Fulbright Doctoral Dissertation Research Award. The author gratefully acknowledges the support of the Fulbright Program and sincerely thanks The University of Texas at Austin for its hospitality and institutional support throughout the fellowship. The author would also like to thank Lior Alon for many fruitful discussions and valuable remarks that significantly improved this manuscript. The author is also grateful to the Massachusetts Institute of Technology (MIT) for its hospitality during a separate research visit. Finally, the author is deeply grateful to Felipe Gonçalves for his valuable suggestions and insightful remarks, and for introducing the author to the theory of crystalline measures.
\subsection*{Competing interest.} The author has no competing interest to declare.

%\newpage

\end{document}